\documentclass[onenumber, plainnumbering,english,11pt]{amsart}
\usepackage[applemac]{inputenc}
\usepackage{bigints}
\usepackage{mathrsfs}
\usepackage[square,numbers]{natbib}
\usepackage{amssymb,amsmath,amscd,mathrsfs,paralist}

\usepackage[mathscr]{eucal}
\usepackage{bm}
\usepackage[T1]{fontenc}
\usepackage{graphicx}
\usepackage{tikz}
\usepackage{cool}
\usepackage[all,cmtip]{xy}
\usepackage[colorlinks,linkcolor=black,urlcolor=black]{hyperref}
\hypersetup{
     colorlinks   = true,
     citecolor    = black}
\usepackage{pgfplots}
\newcommand{\R}{\mathbb R}
\newcommand{\N}{\mathbb N}
\newcommand{\Z}{\mathbb Z}
\newcommand{\C}{\mathbb C}

\newcommand{\Q}{\mathbb Q}

\newcommand{\floor}[1]{{\left\lfloor #1 \right\rfloor}}

\newcommand{\nc}{\newcommand}
\nc{\BCc}{{\mathbb{C}(\wp(z),\wp^\prime(z))}}
\nc{\BC}{{\mathbb C}}
\nc{\BQ}{{\mathbb Q}}
\nc{\BR}{{\mathbb R}}
\nc{\BZ}{{\mathbb Z}}
\nc{\BP}{{\mathbb P}}
\nc{\BN}{{\mathbb N}}
\nc{\BM}{{\mathbb M}}
\nc{\fH}{{\mathfrak{H}}}
\nc{\vp}{{\varepsilon}}\nc{\dpar}{{\partial}}\nc{\al}{{\alpha}}
\nc{\PSL}{PSL(2,\BR)}
\nc{\PS}{PSL(2,\BZ)}
 \nc{\CL}{PSL(2,\BZ/m\BZ)}
 \newtheorem{theorem}{Theorem}[section]
\newtheorem{corollary}[theorem]{Corollary}
\newtheorem{lemma}[theorem]{Lemma}
\newtheorem{proposition}[theorem]{Proposition}
\theoremstyle{definition}
\newtheorem{definition}[theorem]{Definition}
\newtheorem{remark}[theorem]{Remark}
\newtheorem{example}[theorem]{Example}

\DeclareMathOperator{\lcm}{lcm}
\makeatletter
\newcommand\xleftrightarrow[2][]{%
  \ext@arrow 9999{\longleftrightarrowfill@}{#1}{#2}}
\newcommand\longleftrightarrowfill@{%
  \arrowfill@\leftarrow\relbar\rightarrow}
\makeatother

\numberwithin{equation}{section}  
\begin{document}
\title{ Harmonic Analysis of some arithmetical functions}
\author{Ahmed Sebbar \and Roger Gay}
\address{Ahmed Sebbar \\ Chapman University, One University Drive, Orange, CA 92866, USA \\
Univ. Bordeaux, IMB 
   \\UMR 5251, F-33405 Talence, France.}
   \email{sebbar@chapman.edu, ahmed.sebbar@math.u-bordeaux.fr}
   \address{ Roger Gay\\
Univ. Bordeaux, IMB 
   \\UMR 5251, F-33405 Talence, France.}
   
      \email{roger.gay@math.cnrs.fr}
\begin{abstract}
We study three functions which are power series in the variable $z$, Dirichlet series in the variable $s$ and with coefficients given by arithmetical functions. A strong point is to relate these functions to some Hilbert spaces. Three main
ingredients are used: an estimate of Davenport on sums of M\"obius functions, a result of Lucht on convolutions of arithmetical Dirichlet series and the introduction of an operation $\otimes$ on power series, naturally associated with the mentioned Hilbert spaces.

\end{abstract}
\maketitle 
\section { Inroduction}

Several formal trigonometrical expansions of the Analytic Number Theory are of Harmonic Analysis nature. For instance they are periodic or almost periodic Fourier series of their sums. The main goal of the present paper is to prove a corresponding result for three arithmetical functions called 
$   {\mathscr L}_s,\,  {\mathscr M}_s,\,  {\mathscr C}_s   $. The first is the classical polylogarithm function, the second is built from the M\"{o}bius function $\mu(n)$ and the third from the Ramanujan sums. The most saliant results of the paper can be  summarized as follows. We will study some possible links between $   {\mathscr L}_s,\,  {\mathscr M}_s,\,  {\mathscr C}_s              $ by using a theorem of Lucht \cite{Lucht}. For some arithmetical functions, as
 for example $\displaystyle \frac{\sigma_s(n)}{n^s},\, \sigma_s(n)=\displaystyle \sum_{d\vert n}d^s,\, \Re s>0  $ we study the existence of 
 Ramanujan expansion and give  the Ramanujan coefficients. The third objective is to look, from what is called Kubert's identities point of view to the problem, first  solved by Besicovitch \cite{Besico}, of giving an example of a non-trivial real continuous function  $f $ on  $[0, 1]$ which is not odd with respect to the point  $\frac{1}{2}$ and which has the property that for every positive integer $k$
  \[ \sum_{h=0}^k\ f(\frac{h}{k})= 0.\]
 Bateman and Chowla \cite{BC}, \cite{Chowla} gave the two explicit examples of such functions
\begin{align*}
f_1(t)&= \sum_{n=1}^{\infty} \frac{\mu(n)}{n} \cos (2\pi n t)\\
f_2(t)&= \sum_{n=1}^{\infty} \frac{\lambda(n)}{n} \cos (2\pi n t)
\end{align*}
where $\mu$ is the M\"{o}bius function and $\lambda$ is the Liouville's function $\lambda$, defined by $\lambda(1)=1$ and $\lambda(n)= (-1)^j$ if $n$ is the product of $j$ (not necessarily distinct) prime numbers. It is a multiplicative function, closely related to the M\"{o}bius $\mu$
function for they coincide on square-free integers. These two functions share many properties, as we will see in the last section.

We introduce some Hilbert spaces and build Riesz basis from the function $   {\mathscr L}_s$  and determine an biorthogonal basis. The characterizations of the Riesz bases highlight some Dirichlet series as well as some extension of the famous Smith determinant.
We illustrate the Fourier Analysis aspect through the Ramanujan series and their use in the development of arithmetical functions.
The last section briefly presents an opening towards dynamic systems, to emphasize that the path inaugurated by Aurel Wintner, Norbert Wiener and Marc Kac may experience a revival in dynamical systems, as in the conjectures of Chowla and Sarnak.
\section{Arithmetical functions}
Lambert series are, by definition, series of the form
\[ \sum_{n=1}^{\infty}\ a_n\frac{x^n}{1-x^n},\quad a_n \in \BC.\]
They were considered in connection with the convergence of power series. If a series $\displaystyle \sum_{n=1}^{\infty} a_n  $ converges, then the Lambert series converges for all $x\neq \pm 1$. Otherwise it converges for those values of $x$ for which the series $\displaystyle \sum_{n=1}^{\infty} a_n x^n$ converges \cite{Spira}, and the references therein.
In all that follows, it would be of some interest to highlight three equivalences which will be used more or less explicitly in this paper.  We have formally the following diagram  where $f$ and $g$ are two arithmetical functions 
 \begin{equation*}
 f(n)= \sum_{d\vert n}g(d) \iff \sum_{n=1}^{\infty} \frac{f(n)}{n^s}= \zeta(s) \sum_{n=1}^{\infty} \frac{g(n)}{n^s} \iff
 \sum_{n=1}^{\infty} f(n) x^n= \sum_{n=1}^{\infty}g(n) \frac{x^n}{1-x^n}.
 \end{equation*}
 This is exactly the essence of our work: We are constantly moving between three aspects: Arithmetical convolution, Dirichlet series and power series. This is done through  the Riemann zeta function or its inverse.
To illustrate this, we give some examples, some of which will be used and all the definitions will be given,
\begin{enumerate}
\item If $g(n)= \mu(n) $, the M\"{o}bius function, then 
\[  \sum_{n=1}^{\infty}\mu(n) \frac{x^n}{1-x^n}= x.\]
\item If $g(n)= \lambda(n) $, the Liouville function, the associated Lambert series is the Jacobi theta function
\[  \sum_{n=1}^{\infty}\lambda(n) \frac{x^n}{1-x^n}= x+x^4+x^9+x^{16}+\cdots.\]
\item If $\Phi(n)$ is Euler's totient function, then  for $\vert x\vert  $<1
\[\sum_{n=1}^{\infty} \Phi(n)\frac{x^n}{1-x^n}=  \frac{x}{(1-x)^2}. \]
\item If $\displaystyle G_1(x)=  \sum_{n=1}^{\infty}g(n) \frac{x^n}{1-x^n}$ and $ \displaystyle G_2(x)=  \sum_{n=1}^{\infty}g(n) x^n$, then
\[G_1(x)=  \sum_{n=1}^{\infty} G_2(x^n). \]
\end{enumerate}
When $g(n)$ is a known arithmetical function, like $\mu(n)$ or $\lambda(n)$ or $\Phi(n)$,  the previous relations reflect deep arithmetical identities. On the other hand some
elementary functions $g(n)$ can produce non trivial sums. For example if $\displaystyle g(n)= \frac{1}{n} $ and $\displaystyle G_1(x)=  \sum_{n=1}^{\infty}\frac{1}{n}  \frac{x^n}{1-x^n}$,  then 
\[ e^{-G_1(x)}= \prod_{n=1}^{\infty} (1-x^n),\]
a well known function in the theory of partitions.

The notion of Kubert's identity is important for us, before defining it we introduce a fundamental function 
\begin{equation}\label{kubert}
\{t \}= \left\{ \begin{matrix}
t-\floor{t}-\frac{1}{2}& {\rm if}\quad  t\neq \floor{t}\\
\\
0& {\rm if}\quad  t= \floor{t}
\end{matrix}
\right.
\end{equation}
  admitting the Fourier expansion
\[\{t\}= -\frac{1}{\pi} \sum_{n=1}^{\infty} \frac{\sin(2\pi mt}{m},       \]
which extends into a formal summation expansion 
\begin{equation*}
\sum_{n=1}^{\infty} \frac{a_n}{n}\{n t\}= -\frac{1}{\pi} \sum_{n=1}^{\infty} \frac{A_n}{n}\sin (2\pi n t),
\end{equation*}
where
$\displaystyle A_n= \sum_{d\mid n} a_d.$
  This reveals a property of the sequence $(\{nt\})_{n\geq 1} $, closely related to the main concern of this paper. We have the well known result 
\begin{equation}\label{Chowla3}
\int_0^1\left(\{rt \} \{ st\}   \right)\,dt= \frac{\gcd(r,s)}{12 \lcm(r,s)}= \frac{\gcd(r,s)^2}{12 rs}.
\end{equation}
Another example that we will meet  is the expansion, $t\notin 2\pi \Z        $
\begin{equation}\label{kubert2}
\log\left(2\left \vert \sin \frac{t}{2} \right \vert\right)= -\sum_{n=1}^{\infty} \frac{\cos nt}{n}
\end{equation}
which leads to the formal identity, for irrational $t$
\[ \sum_{n=1}^{\infty} \frac{c_n \log (2\vert \sin n\pi t\vert))}{n}= -\sum_{n=1}^{\infty}  \frac{G_n \cos (2n\pi t)}{n},                      \]
where again  $ \displaystyle G_n= \sum_{d\vert n} c_d      $. The validity  of this equality has been discussed by Davenport in \cite{D} and also  by Chowla in \cite{Chowla2}, who observed that
\begin{equation}\label{Chowla33}
\int_0^1 \log\left(2\vert \sin r\pi t  \vert\right)       \log\left(2 \vert \sin s\pi t  \vert\right)\,dt= \frac{\pi^2}{12}\, \frac{\gcd(r,s)^2}{rs}.
\end{equation}
The formulas \eqref{Chowla3}  and \eqref{Chowla33} are named Franel integrals in \cite{SebGay}.
Beside Number Theory, the functions \eqref{kubert} and  \eqref{kubert2} appear in Fourier and Harmonic Analysis where \eqref{Chowla3} and \eqref{Chowla33} find an interpretation. To give the mean idea we cite the following fact:  the sequence of functions
\[1, \{t\}, \{2t\}, \cdots \{nt\},\cdots  \]
  is a basis for the Hilbert space  $ (L^2([0, \frac{1}{2}), dt)$,  $dt$  being the Lebesgue measure. This kind of results, with very interesting connections with questions in Number Theory, appeared in \cite{Wintner}  and \cite{Hartman}.\\  
  Another point of view, which we only briefly evoke here and also  in the section \eqref{lastsection}, is the following: We fix a positive integer $p$ and define on the unit interval the $p$-Bernoulli map, an extension of \eqref{kubert}, the function
\[\psi_p(x)= px-\floor{px}, \quad \{x \}_p = \psi_p(x)- \frac{1}{2}\]
which admits the Fourier series expansion
\begin{equation}\label{Fourier}
 \{x \}_p= -\sum \frac{2\cos(2n\pi x-\frac{1}{2}p\pi)}{(2n\pi)^p}.
 \end{equation}
We look at $\psi_p(x) $ as a one-dynamical system on the space $(0, 1) $, as in \cite{Hata}. The associated  Perron-Frobenius operator $P_{\psi_p}$ is defined by
\[\left(P_{\psi_p}u \right)(x)= \sum_{y\in \psi^{-1}(x)} \frac{u(y)}{\vert \psi'(y)\vert}= \frac{1}{p} \left\{u(\frac{x}{p})+u(\frac{x+1}{p})+\cdots+ u(\frac{x+p-1}{p}) \right\}.           \]
If $u$ is an eigenvector of $P_{\psi_p} $, associated to the eigenvalue $\lambda$, then
 \begin{equation} \label{kubert1}
\lambda pu(px)= u(x)+u(x+\frac{1}{p})+\cdots+ u(x+\frac{p-1}{p}).      
\end{equation}
We see that if, for example, $\lambda = 1$, the eigenfunctions satisfy certain functional equations, similar to those satisfied by the function $\log \Gamma$.
We give few  fundamental examples: 
\begin{enumerate}
\item Bernoulli polynomials  are given by
\[\frac{te^{tx}}{e^t-1}= \sum_{n=0}^{\infty} B_n(x) t^n,\quad B_0(x)= 1,\;\;B_1(x)= x-\frac{1}{2},\;\;  B_2(x)= x^2-x+\frac{1}{6}, \cdots \]
and they satisfy
\[\sum_{n=0} ^{\infty}\left(B_n(x)+\cdots+B_n(x+\frac{k-1}{k} \right)t^n= k \sum_{n=0} ^{\infty} (\frac{t}{k})^nB_n(kx).              \]
So  the eigenvalues are $\lambda= k^{-n}$.\\
\item Hurwitz zeta function, defined for $\Re s>1  $ by
\[\zeta(s, x)= \sum_{n=0} ^{\infty} \frac{1}{(x+n)^s} \]
for which we have
\[\zeta(x, s)+\cdots+\zeta(x+\frac{k-1}{k}, s)= k^s\zeta(kx, s) \]
and the eigenvalues are  $\lambda= k^{s-1}$. It satisfies for $ \Re s>\frac{1}{2} $ the Franel type integral  \cite{Mikolas}
\[\int_0^1 \zeta(\{ax\}, 1-s)  \zeta(\{bx\}, 1-s) \,   dx = \frac{2\Gamma^2(s) \zeta(2s)}{(2\pi)^{2s}}  \left(\frac{\gcd(a,b)}{\lcm(a,b)} \right)^s .    \]
The integral diverges for $ \Re s\leq \frac{1}{2} $.
\item The polylogarithm function defined (for $ \vert z\vert<1, \Re s\geq 1$ or  $ \vert z\vert \leq 1, \Re s> 1$) by 
\[  {\mathscr L}_s(z)= \sum_{ n \geq 1} \frac{z^k}{k^s}.  \]
\end{enumerate}
For $s=k$ an integer the polylogarithm function is related to the Bernoulli polynomial $B_k(X)$ by
\[B_k(\floor \theta) = -\sum_{n\neq 0} \frac{e^{2i \pi n \theta}}{n^k},\]
which is just \eqref{Fourier} when $k= p=1$. In order to study their relation to the Perron-Frobenius operator we introduce a new notion.
\begin{definition}
 According to Kubert \cite{Kubert}, \cite{Milnor}  we say that a function $f(x)$, where $x$ varies over $\Q/\R $ or $\R/\Z$, satisfies a Kubert identity if it verifies 
 the functional equations
\begin{equation*}
f(x)= m^{s-1} \sum_{k=0}^{m-1} \left(\frac{x+k}{m}\right)\quad \quad \quad \quad \quad \quad \quad \quad\quad \quad(\star_s)
\end{equation*}
for every positive integer $m$. Here $s$ is some fixed parameter.
\end{definition}
It is apparent that this definition is more restrictive than the one given for eigenfunctions of the Perron-Frobenius operator. The derivative of a differentiable function satisfying $ (\star_s)$ satisfies  $ (\star_{s-1})$. A very instructive example is given by the following example: From the polynomial relation
\[X^n-1= \prod_{\eta^n=1}(\eta X-1) \]
we deduce that \[e^{2i\pi nx} -1=  \prod_{k= 0}^{n-1}(e^{2i\pi(x+\frac{k}{n})} -1),\quad x\in \Q/\Z,\; x\neq 0\]
so that if
$f(x):= \log \vert e^{2i\pi  x} -1\vert $,  then
\begin{equation}\label{smallexample}
f(nx) = \sum_{k=0}^{n-1}  \log \vert 2\sin \pi (x+\frac{k}{n})  \vert=  \sum_{k=0}^{n-1}  f(x+\frac{k}{n}). 
\end{equation}
This property of the function $\displaystyle f(x):= \log \vert e^{2i\pi  x} -1\vert $ is connected with Franel type equality \eqref{Chowla33}.\\
Let ${\mathscr K}_s,\,s\in \C$, be the complex vector space of all continuous maps $f: (0, 1) \to \C  $ which satisfy the Kubert identity $ (\star_s)$ for every positive integer $m$ and every $x$ in $(0,1)$.  It is easy  to see directly  that  the function ${\mathscr L}_s(z)$ satisfies the relation ($\star_s$). More precisely
 (\cite{Milnor}, p.287)
\begin{theorem} The complex vector space ${\mathscr K}_s$ has dimension 2, spanned by one even element ($f(x)= f(1-x)$) and one odd element ($f(x)= -f(1-x)$). Each $f(x)\in {\mathscr K}_s $ is real analytic.
\end{theorem}
This is an interesting interpretation of an important result. In fact if 
\[  {\mathfrak l} (x)= {\mathscr L}_s(e^{2i\pi x})\]
we should have, according to this theorem, a linear combination
\[{\mathfrak l} (x)= A_s \zeta(1-s, x)+ B_s \zeta(1-s, 1-x). \]
The values of the coefficients are given in (\cite{Milnor}, p.308)
\[A_s= \frac{i(2\pi)^se^{-\frac{i\pi s}{2}}}{2\Gamma(s) \sin(\pi s)}, \quad  B_s= \frac{-i(2\pi)^se^{\frac{i\pi s}{2}}}{2\Gamma(s) \sin(\pi s)}.  \]
This is precisely an another formulation of Lerch's transformation formula for the function
\[\Phi(z,s, \nu)= \sum_{n=0}^{\infty}\frac{z^n}{(n+\nu)^s},\quad \vert z\vert <1,\; \nu\neq 0, -1,-2,\cdots \]
which is (\cite{Higher}, p.29)
\begin{align*}
\Phi(z,s, \nu)=&\\
 iz^{-\nu} (2\pi)^{s-1}& \Gamma(1-s)\left\{e^{\frac{-i\pi s}{2}} \Phi[e^{-2i\pi \nu},1-s, \frac{(\log z)}{2i\pi} ]-
e^{i\pi(\frac{ s}{2}+2\nu)} \Phi[e^{2i\pi \nu},1-s, 1- \frac{(\log z)}{2i\pi} ] \right\}
\end{align*}
and which reduces to the functional equation of the Riemann zeta function when $z=1, \nu=0 , \Re s>1$.
\begin{remark}
The theorem of Milnor asserts that every function in the space $ {\mathscr K}_s$ is real analytic.  These functions are eigenfunctions corresponding to the eigenvalue $\lambda=1$ of the Perron-Frobenius operator. However the later has  eigenfunctions corresponding to the eigenvalue $\lambda= \frac{1}{2}$ which are continuous but nowhere differentiable. As mentioned  in (\cite{Hata}, p.361) the Tagaki function (or the blancmange function)  $T(x)$ is an example of a such function. This function is defined by
\[T(x)= \sum_{n=1}^{\infty}  \Psi(2^nx)-\frac{1}{2} \]
where
 \[\Psi(x) = \inf \{ \vert x-n\vert , n \in \Z\}= \left \vert x-2\floor{\frac{x+1}{2}}\right \vert \]
  is the sawtooth function, periodic  of period $1$.
\end{remark}
 
\section{Three power series}
The essential of the analytic properties  of   the polylogarithm function $\displaystyle {\mathscr L}_s(z)= \sum_{ n \geq 1} \frac{z^k}{k^s}  $                    come from the integral representation 
\[{\mathscr L}_s(z)= \frac{z}{\Gamma(s)} \int_0^{\infty} \frac{t^{s-1}}{e^t-z}\, dt, \quad\Re s>0,\;z\notin (1,\infty).\]
Let $\displaystyle \vartheta= z\frac{d}{dz}$  be Boole's differential operator. We define an inverse  of $\vartheta$  by \[\displaystyle  \vartheta^{-1} f(z)= \int_0^z  f(u) \frac{du}{u}\] defined on the class of analytic functions near the origin, and vanishing at the origin. For $s=n$ a positive integer we have the symbolic representation as an iterated integral
\begin{equation}\label{iterated}
\displaystyle  {\mathscr L}_n(z)= \vartheta^{-n}\frac{z}{1-z}.\end{equation}
To define the next function we recall first the definition of the M\"{o}bius arithmetical function 
\[\mu(n)= \mu_n= \begin{cases}
1   \text{ if} \,\; n= 1\\
\\
 0 \text{ if } \, n\text { has one or more repeated prime factors}\\
 \\
(-1)^k  \,\text {if}\, \, n\,\, \text{is the product of } k \, \text{prime\, factors}.
\end{cases}
\]
The importance of the M\"{o}bius function  lies in the following inversion
\begin{equation}\label{Mobinv}
 f(n)= \sum_{d\vert n}g(d) \iff g(n)= \sum_{d\vert n}f(d)\mu(\frac{n}{d})=  \sum_{d\vert n}f(\frac{n}{d} )\mu(d).
\end{equation}

The generalized M\"{o}bius inversion may be formulated as follows:
If t varies on the half-line t > 0, and if either
$\displaystyle g(t)= O(t^{-1-\eta})$
 holds for some $\eta>0$ or
 $\displaystyle h(t)= O(t^{-1-\delta}) $
  holds for some $\delta>0$ then
  \[h(t)= \sum_{n=1}^{\infty}g(nt) \]
is equivalent to
\[g(t)= \sum_{n=1}^{\infty} \mu(n) h(nt).\]
 The main objective of this section is the study the relationships between  the three functions ${\mathscr L}_s(z), {\mathscr M}_s(z)$ and ${\mathscr C}_s(z)$ defined by the following power series:
\begin{enumerate}
\item  $\displaystyle {\mathscr L}_s(z) = \sum_{ k \geq 1} \frac{z^k}{k^s},      \; \vert z \vert \leq 1, \quad \Re s >1 $ or $ \vert z \vert <1, \; \Re s \geq 0       $,
\item  $ \displaystyle {\mathscr M}_s(z) =  \sum_{ k \geq 1}\mu_k \frac{z^k}{k^s}, \; \vert z \vert \leq 1  \; \Re s >1 $ or $  \vert z \vert <1, \quad \Re s \geq 0           $.
This series is most known when $s=0$ and $\vert z\vert<1 $. It amounts to the series 
$\displaystyle \sum_{k \geq 1} \mu_k z^k$ studied by Bateman and Chowla \cite{BC}, \cite{Chowla}. They use the crucial estimates for sums of M\"obius functions values, due to Davenport \cite{D}: For every $A>0$, there exists a constant $D(A)$ such that, uniformly for $\vert z\vert \leq1$
\begin{equation}\label{Bateman}
\left\vert \sum_{0<j\leq x}\mu(j)z^j\right\vert \leq D(A) \log(x+1)^{-A}.
\end{equation}
\item  $\displaystyle {\mathscr C}_{s,l}(z) = \sum_{k \geq 1} c_k(l)  \frac{z^k}{k^s} \quad \vert z \vert \leq 1  \quad \Re s >1 $ or $ \vert z \vert <1, \quad \Re s \geq 0   $,
\end{enumerate}
where  $c_k(l) $ is the Ramanujan  sum
\[\displaystyle c_q(n)= \sum_{\substack {a=1\\(a,q)=1}}^n e^{2 i \pi \frac{an}{q}}=     \sum_{\substack {a=1\\(a,q)=1}}^n \cos(2\pi \frac{an}{q}).\] 
As we will see the series $ \displaystyle {\mathscr C}_{s,l}(z) = \sum_{k \geq 1} c_k(l)  \frac{z^k}{k^s}$ is most known only when $z=1$ and $ \Re s>1 $. Its sum was given by Ramanujan, and simplified methods were found by Estermann and others.\\
 For fixed $n$, $c_q(n)$  is a multiplicative function: if $q_1,q_2$ are relatively prime, then
\[c_{q_1}(n)c_{q_2}(n)= c_{q_1 q_2}(n).\]
Moreover $c_q(n)$ is a periodic function of $n $ with period $q$.
When $(m,k)=1$ we have $c_k(m) = \mu_k$, and when $(m, k) = k$ we have $c_k(m)= \Phi(m)$, $\Phi$ being the Euler's totient function, with for every
positive integer $N$,  $\Phi(N)$ is the number of positive integers less than or equal to $N$ and relatively prime to $N$. More generally H\"{o}lder \cite{Holder} showed that Ramanujan's sum can also be expressed in closed form as follows:
\[c_k(m)= \frac{\Phi(m)}{\Phi \left(\frac{m}{(k,m)}\right)} \mu\left(\frac{m}{(k,m)}\right). \]
Three well known properties of the $\Phi$-function are important for further extension. For every positive integer $N$
\begin{align}\label{totient}
 N&= \sum_{d\vert N} \Phi(d) \nonumber\\ 
 \Phi(N)&= N \prod_{\substack{ p\vert N\\ 
p\,{\rm prime}}}          \left(1-\frac{1}{p} \right).
\end{align}
An important property of the Ramanujan coefficients is their orthogonality, that can be used to compute the Ramanujan coefficients
\begin{lemma}[Orthogonality relations]
Let $\Phi$ the Euler's totient function, then
\[\lim_{x \to +\infty} \frac{1}{x} \sum_{n\leq x} c_r(n) c_s(n)= \Phi(r)\]
if $r=s$ and zero otherwise. More generally
\[  \lim_{x \to +\infty} \frac{1}{x} \sum_{n\leq x} c_r(n) c_s(n+h)=   c_r(h)    \]
if $r=s$ and zero otherwise.
\end{lemma}
 The functions ${\mathscr M}_s,    {\mathscr L}_s $, tough different in nature, share the same difference-differential equation
\begin{equation}\label{F-equation}
z\frac{\partial}{\partial z}f(z,s)= f(z,s-1),
\end{equation}
but the series  $ \displaystyle {\mathscr M}_s(z)$ does not seem to have attracted much attention. For the particular case $z=1  $ and $\Re s>1   $ we have with $\displaystyle \sigma_{s-1}(n)= \sum_{d\vert n}d^{s-1}$
\[{\mathscr L}_s(1) = \zeta(s),\quad {\mathscr M}_s(1) = \frac{1}{\zeta(s)},\quad {\mathscr C}_{s,l}(1) = \frac{\sigma_{s-1}(n)}{n^{s-1}\zeta(s)}.\]
The last equality can be understood in the framework of Ramanujan-Fourier series: Given an arithmetical function $a:\BN \to \BC$, the Ramanujan-Fourier series for $a$ is 
the series 
\[a(n)= \sum_{q=1}^{\infty} a_q c_q(n),\quad n\in \BN.\]
The coefficients can be computed by using the previous  orthogonality relations,
but we will be concerned by  a different kind of approach.\\
Let  $f,\, g: \N \to \C $ two arithmetical functions. The Dirichlet convolution product of  $f, g$ is defined by
\[f\star g(n)= \sum_{d\vert n} f(d) g(\frac{n}{d}),\; n\in \N.\]
For example  if $1: n\to n      $ is the identity arithmetical function,   the inversion formula \eqref{Mobinv} is just
\[f=g\star1 \iff g=f\star \mu. \]
To study the expansion in Ramanujan series, or to find relations between the three series $  {\mathscr L}_s(z), {\mathscr M}_s(z) $  and $  {\mathscr C}_{s,l}(z) $ we will make use of a result of H. Delange \cite{Delange} and a result of L. Lucht \cite{Lucht}. 
\begin{theorem}[Delange]
Suppose that
\[f(n)= \sum_{d\vert n} g(d)= (g\star 1)(n) \]
and that
\begin{equation}\label{Delange1}
 \sum_{n=1}^{\infty}  2^{\omega(n)}\frac{\vert g(n)\vert}{n} <\infty,
\end{equation}
where $\omega(n)$  is the number of distinct prime divisors of n. Then f admits a Ramanujan expansion with
\[ \hat{f}(q)= \sum_{m=1}^{\infty} \frac{g(q m)}{qm}. \]
More precisely for each $n$ the sum
\[\sum_{q=1}^{\infty} \hat{f}(q)c_q(n)  \]
is absolutely convergent and is equal to  $f(n)$.
\end{theorem}
We can weaken this statement by observing that  from the prime decomposition of  a positive integer $n= \displaystyle \prod_{i=1}^kp_i^{e_i},\, e_i\geq 1$, the number of 
divisors $d(n)$ and the number of prime divisors $\omega(n)$ are related by
\[d(n)= (e_1+1)+\cdots+(e_k+1)\geq 2^{\omega(n)}. \]
So the condition \eqref{Delange1} can be weakened by asking only
\begin{equation}\label{Delange2}
 \sum_{n=1}^{\infty}d(n)\frac{\vert g(n)\vert}{n} <\infty.
\end{equation}
To give an application of this theorem we begin by a review of some additional properties of Ramanujan sums. First of all for $z=1$ the sum is the celebrated formula of Ramanujan \cite{Ramanujan} (p.199): For  $k > 1$ and  $s\in \C\setminus \{1\}$ with $\Re s>0$
\begin{equation}\label{ramanujan}
\sum_{m=1}^{\infty} \frac{c_k(m)}{m^s}= \zeta(s) \sum_{d\vert k} d^{1-s}\mu(\frac{k}{d}),
\end{equation}
 and even for $s=1$ we have \cite {Ramanujan} (p.199)
\[ \sum_{m=1}^{\infty} \frac{c_k(m)}{m}= - \sum_{d\vert k}\mu(\frac{k}{d}) \log d= -\Lambda(k), \]
where $k> 1$  and $\Lambda(k)$ is Mangoldt's function
\[\Lambda(n)= \left\{\begin{matrix}
\log p& {\rm if}\; n= p^m\; {\rm for\,some}\, m\geq 1\\
\\
0& {\rm else}
\end{matrix}
\right.
\]
Another formula of Ramanujan  \cite {Ramanujan} (p.185) is
\begin{equation}\label{Ramanujan2}
 \sum_{k=1}^{\infty} \frac{c_k(m)}{k^s}= \frac{\sigma_{1-s}(m)}{\zeta(s)},
\end{equation}
valid for $\Re s> 1$ and also for $s = 1$.  In this case the sum vanishes. According to Ramanujan this assertion is  equivalent to the Prime Number Theorem.\\
This formula results directly from Delange's theorem.
The next role will be played by the following lemma. 
 \begin{lemma}[Romanoff]
  If $ k <n$ and  $f(u)$ is any function defined on the positive integers then
\begin{equation}\label{Romanoff}
\sum_{d\vert n} \mu(\frac{n}{d}) f((d,k))= 0.
\end{equation}
 \end{lemma}
This Lemma  can be proved, for example, by establishing that 
\[ \sum_{d\vert n} \mu(\frac{n}{d}) f((d,k))= \sum_{\delta} f(\delta) \sum_{\substack{d\vert n \\ (d,k)=\delta}} \mu(\frac{n}{d}).\]
This lemma is very important of the following reason. Let $(x_n)$ be a sequence in a Hilbert space $\mathcal H$ with an inner product
$ \langle , \rangle$ and let $g: \N\times \N \setminus \{(0, 0)\}\to \C$ be an  arithmetical function such the $\langle x_n, x_m \rangle= g(n,m) $. Then 
the sequence \[y_n= \sum_{d\vert n}  \mu(\frac{n}{d}) x_d            \] is an orthogonal sequence.
 Another important result is the following
\begin{theorem}[Lucht]\label{th1}
Let $g: \N\to \C$ be an arbitrary arithmetical function. Then the following assertions are equivalent: 
\begin{enumerate} 
\item The series $\hat{g}(k)= \displaystyle \sum_{n \geq 1} g(n)c_n(k) $ converges (absolutely) for every $k\in \N^*$ and determine an arithmetical function $\hat{g}$.
\item The series  $\gamma(k)= k\displaystyle \sum_{n \geq 1} \mu(n) g(kn) $  converges (absolutely) for every $k\in \N^*$ and determine an arithmetical function  $\gamma$.
\end{enumerate}
In the case of convergence we have the convolution products $\gamma = \mu \star {\hat{g}} $ or $1\star \gamma= \hat{g}   $.
\end{theorem}

As a first application we take  $\displaystyle g(n)=g_{z,s}(n) = \frac{z^n}{n^s} $, with $\vert z \vert  \leq 1$, $\Re s >1 $ or $\vert z \vert <1,\; \Re s\geq 0$  to obtain by the  theorem of Lucht:
 \[\hat{g}(k) =  \hat{g}_z(k)= \sum_{n \geq 1} c_n(k) \frac{z^n}{n^s} = {\mathscr C}_{s,k}(z). \]
 Hence
 \[\gamma (k) = \gamma_{z,s}(k) = \frac{1}{k^{s-1}} \sum_{n \geq 1}\mu_n \frac{z^{nk}}{n^s} = \frac{1}{k^{s-1}} {\mathscr M}_s(z^k).\]
For $s=1$ the uniform convergence on the unit closed disk of the series  $\displaystyle \sum_{n \geq 1} \frac{\mu(n)}{n}z^n$ results from the uniform convergence on $\BR$ of the series $\displaystyle \sum_{n \geq 1} \frac{\mu(n)}{n}e^{2i\pi \theta} $, resulting from  \eqref{Bateman} and the  maximum principle. For the sake of completeness we give all the details of this important result. The following lemma, due to Cahen and Jensen, is classical \cite{HR}:
\begin{lemma}\label{Cahen}
If the Dirichlet series $\displaystyle f(s) = \sum_{k=1}^{\infty} a_ke^{-\lambda_k s}$
  is convergent for $s_0$, then it is also convergent for any $s$ in the cone $\vert \arg(s-s_0)\vert \leq \alpha < \frac{\pi}{2}$\,{\rm (Stolz\,angle)}. Furthermore the series is uniformly convergent on any compact set of the cone, as well as any of its derivatives and \[ \displaystyle f^{(n)}(s) = (-1)^n \sum a_n \lambda_k^n e^{-\lambda_n s}.\]
\end{lemma}
Accordingly
 \[\lim_{s \to 1} {\mathscr M}_s(e^{2i \pi \theta})={\mathscr M}_1( e^{2 i \pi \theta)}).\]
 The Ramanujan-Fourier transform of the arithmetical function  $\displaystyle g(n)= \frac{ z^{n}}{n} $ is  
 \[\displaystyle \hat{g}(k) = \sum_{n \geq 1} c_n(k) \frac{z^n}{n}\] which converges for $ \vert z \vert <1$.
 It also converges for $\vert z \vert =1$. In fact 
 \[\displaystyle \gamma(k) = k \sum_{n \geq 1} \mu(n) \frac{ z^{nk}}{nk} = \sum_{n \geq 1} \mu(n) \frac{ z^{nk}}{n} \]
  converge uniformly on the closed unit disc by Davenport's estimate \eqref{Bateman}. By using \eqref{th1} we obtain the convergence of the series $\hat{g}(k) $ with $\gamma(k) ={\mathscr M}_1(z^k) $, and  finally 
\[
\sum_{n \geq 1} c_n(l) \frac{z^n}{n}  = \sum_{ d \vert l} {\mathscr M}_1(z^d)
\]
for $\vert z \vert =1$, by using a M\"obius inversion of $\gamma = \mu \star \hat{g}$. The maximum principle asserts that this equality continues to be valid even for $\vert z \vert \leq 1$.\\
In the same vein we have:
\[
{\mathscr C}_{s, l} (z) =  \sum_{d \vert l} d^{1-s} {\mathscr M}_s(z^d),
\]
which is an interesting link between the two series $ {\mathscr C}_{s, l} (z) $ and $ {\mathscr M}_{s} (z) $. 
The following lemma is elementary and it is just a variation of \eqref{smallexample}.
\begin{lemma}
We have 
\[\sum_{1 \leq h \leq m} e^{2 i \pi \frac{nhk}{m} }= \left\{
    \begin{array}{ll}
        m & {\rm if }\; m\vert nk \\
        \\
        0& { \rm otherwise.}
    \end{array}
\right.,
\]
or, equivalently, for all $n,m \in \N^*$ and $w,z \in \C$, we have
$$
\sum_{w: w^m=z} w^n = m z^{ \frac{n}{m}}
$$ if $m\vert n$ and zero otherwise.
\end{lemma}
\begin{remark} The meaning of   $nk \in m \N^*$ is the following: We first observe that $k\Z \cap m \Z= \lcm(k,m) \Z$, hence
\[ nk\in  \lcm (k,m) \N^*  \iff n \in \delta(k,m) \N^*,\quad \displaystyle \delta(k,m): = \frac {\lcm(k,m)}{k}\]
and
\begin{equation}
\begin{aligned}
\sum_{1 \leq h \leq m} e^{2 i \pi \frac{nhk}{m} } = \left\{
             \begin{array}{ll}
m & {\rm if}\; n \in \delta(k,m) \N^*\\
\\
0& {\rm \text otherwise}
\end{array}
\right.
\end{aligned}
\end{equation}
 \end{remark}

If we choose $\displaystyle z= e^{2i \pi \frac{h}{m}}$ with some fixed  $m \in \N^*$ and $ 1 \leq h \leq m$, and denote simply 
\[\displaystyle g_{z,1} (n)= g_{\frac{h}{m}} (n) =\frac{e^{2 i \pi \frac{nh}{m}}}{n},\quad \displaystyle \gamma_{z,1}= \gamma_{ \frac{h}{m}}\] we get
\[\gamma _{ \frac{h}{m}} (k) =  k \sum_{n \geq1} \mu(n) g_{ \frac{h}{m}}(kn) =k \sum_{n \geq 1} \mu(n) \frac{ e^{2i \pi \frac{nkh}{m}}}{nk}   ={\mathscr M}_1( e^{2 i \pi \frac{hk}{m}})\]
 and 
\[ \sum_{ 1 \leq h \leq m} \gamma_{\frac{h}{m}} (k) =  m \sum_{ n \in \delta(k,m) \N^*} \frac{ \mu(n)}{n}
\]
or
\[
\sum_{ 1 \leq h \leq m} {\mathscr M}_1 (e^{2i \pi \frac{hk}{m}} ) = m \sum_{ n \in \delta(k,m) \N^*} \frac{\mu(n)}{n}= 0.
\]
As the argument usually used in the proof of the last equality is very present in the present study we give a slightly more general result. We are going to show the following lemma
\begin{lemma}\label{subseries}
 We have  for $\Re s>1$
\begin{equation} \label{subseries2}
\sum_{ n \geq 1} \frac{ \mu(qn)}{n^s}= \frac{\mu(q) q^s}{\Phi_s(q) \zeta(s)}, 
\end{equation}
with
\[ \Phi_s(q)= q^s \prod_{p\vert q}(1-p^{-s})=  \sum_{d\vert q}d^s \mu(\frac{q}{d}).\]
In particular
\begin{enumerate} 
\item $\displaystyle  \sum_{ n \geq 1} \frac{ \mu(qn)}{n} = 0 $ for every $q \in \N^*$.\\
\item  For $d \in \N^*$ we have  $\displaystyle  \sum_{n \in d \N^*} \frac{\mu(n)}{n} =0$
\end{enumerate}
\end{lemma}
Indeed
\[\frac{1}{\zeta(s)}= \sum_{ n \geq 1} \frac{ \mu(n)}{n^s}= \prod_{p}(1-p^{-s})= \prod_{p\vert q}(1-p^{-s}).\prod_{p\nmid q}(1-p^{-s})=  \frac{ \Phi_s(q)}{q^s} 
  \sum_{ n \geq 1,\,(n, q)=1} \frac{ \mu(n)}{n^s}  \]
  \[ \frac{ \Phi_s(q)}{\mu(q)q^s} 
  \sum_{ n \geq 1,\,(n, q)=1} \frac{ \mu(qn)}{n^s}.  \]
  The lemma is  obtained by using the lemma \eqref{Cahen}, or the  classical Ingham's Tauberian theorem \cite{Korevaar} (p.133).
One can show that
\[  \sum_{ n \in \delta(k,m) \N^*} \frac{\mu(n)}{n}= \mu_{\delta(k,m)}  \lim_{ s \to 1} \left( \frac{1}{ \zeta(s) \prod_{ p \vert \delta(k,m)}} (1- \frac{1}{p^s})\right)^{-1} =0.\]
It is easily seen that $\displaystyle \sum_{1 \leq h \leq m} \gamma_{\frac{h}{m}} (k) =0$. Hence
\[
\sum_{1 \leq h \leq m} {\mathscr M}_1( e^{2i \pi \frac{hk}{m}})=0.
\]
From the relation  $\displaystyle \mu \star  \sum_{1 \leq h \leq m} \hat{g}_{ \frac{h}{m}} =0$ we conclude again, by Cahen-Jensen lemma and M\"obius inversion, that
\[
\lim_{s \to 1} \; \sum_{j \geq 1} \frac{c_{jm}(k)}{j^s} = \sum_{j \geq 1} \frac{c_{jm}(k)}{j}=0.
\]
Indeed  we have $\displaystyle   \sum_{1 \leq h \leq m} \hat{g}_{ \frac{h}{m}} =0$, so for every  $k \in \N^*$
\[
\lim_{s \to 1} \; m^{1-s} \sum_{j \geq 1} \frac{c_{jm}(k)}{j^s} = \lim_{s \to 1} \; \sum_{j \geq 1} \frac{c_{jm}(k)}{j^s} = 0
\]
since
\[
\sum_{1 \leq h \leq m} \sum_{n \geq 1} \frac{c_n(k)}{n^s} e^{2 i \pi \frac{nk}{m}} = m^{1-s} \sum_{ j \geq 1} \frac{ c_{jm}(k)}{j^s}.
\]
We also have
\[
 \lim_{s \to 1} \sum_{1 \leq h \leq m} {\mathscr C}_{s,l} (e^{2i \pi \frac{ h}{m}} ) = \sum_{ d \vert l} \left(\sum_{1 \leq h \leq m} {\mathscr M}_1(E^{2i \pi \frac{ hd}{m}})\right) =0
 \]
 by using
 \[
{\mathscr C}_{s, l} (z) =  \sum_{d \vert l} d^{1-s} {\mathscr M}_s(z^d).
\]

This lemma, applied to  $\displaystyle \gamma_{ \frac{h}{m}(k)} = {\mathscr M}_1(e^{2i\pi \frac{hk}{m}})$ and with   $\displaystyle \delta(k,m) = \frac{\lcm (k,m)}{k}$, gives
\[
\sum_{1 \leq h \leq m} {\mathscr M}_1(e^{2i \pi \frac{hk}{m}}) = m \sum_{j \geq1} \frac{ \mu_{j\delta(k,m)}}{j \delta(k,m)} = \frac{m}{\delta(k,m)}  \sum_{j \geq1} \frac{ \mu_{j \delta(k,m)}}{j} = 0.
\]
Thus we have a nontrivial example of a function solving the initial  Besicovich question.
A natural question suggested by the lemma \eqref{subseries} is: Find all the sequences  $a= (a_n)_{n \geq1}$ satisfying  the relations $\displaystyle \sum_{j \geq 1} a_{jm}=0 $ for every $m \in \N^*$. According to \cite{H.Q} (p.294) if
\[\sum_{n=1}^{\infty} \vert a_n\vert< \infty, \quad \sum_{n=1}^{\infty} a_{jn}= 0,\, j=1,2,\cdots \]
then  $a_n=0$ for every $n\geq 1$. As was pointed out in  \cite{H.Q} (p.294) the absolute convergence is necessary. We give here an example \cite{McCarthy}   (p.219) completing 
\eqref{subseries2} and showing the necessity of the absolute convergence. It is
\begin{equation}\label{subseries3}
\sum_{\substack {m=1\\(m,n)=1}}^{\infty} \frac{\vert \mu(m)\vert}{m^s}= \frac{n^s \zeta(s)}{\psi_n(s)\zeta(2s)},
\end{equation}
with
\[ \psi_n(s)= \sum_{d\vert n} d^s \vert\mu(\frac{n}{d}) \vert.\]
\\
It is interesting  to consider the following question later a more general question: For which sequences $a = (a_n)_{ n \geq 1} \in l^2 $ we have
\[
\lim_{\substack{ s \to 1\\  \Re s>1}}  \frac{1}{m^{ s-1}} \sum_{n \geq 1} a_n \left ( \sum_{j \geq1} \frac{\mu_{j \delta(m,k)}}{j^s} \right )=0.
\]
As we saw in \eqref{subseries} we have  $\displaystyle \lim_{\substack{ s \to 1\\  \Re s >1}}  \sum_{j \geq1} \frac{\mu_{j \delta(m,k)}}{j^s} =0$ for every $k \geq 1$. 
 \section{Three examples of Hilbert spaces}
\subsection{Preliminaries}
We propose to introduce a binary operation $ \otimes $  to combine two  power series. Let $D$ be the open unit disk and
\[ H_0^2 (D)= \left\{\sum_ {n \geq 1} a_n z^n,\, a_n\in \BC,\, \sum_ {n \geq 1} \vert a_n \vert^2 <\infty \right\}. \]
Let ${\bf H}$ the space of functions defined almost everywhere on $\BR$, odd and 2-periodic, such that  $f_{\vert (0,1)}\in  L^2 (0,1) $ (that is we will consider only  Fourier sine series expansion of any  $f\in L^2 (0,1)$)  and finally let \[ H^2= \left\{\sum_ {n \geq 1} \frac {a_n} {n ^ s},\, a_n\in \BC, \, \sum_ {n \geq 1} \vert a_n 
\vert^2 <\infty \right\}. \] It is remarkable that from the Hilbert point of view these three spaces are isomorphic but the analytical properties are different but enrich each other. The power series  ${\mathscr L}_s $ and $ {\mathscr M}_s $ belong to  $H_0^2 (D)$ provided that $\Re s> \frac{1}{2} $ and will play a preponderant role. We need the following definition \cite{Spira}
\begin{definition}
Let $\displaystyle A(z)= \sum_{p=1}^{\infty} a_pz^p  $ and $\displaystyle B(z)= \sum_{p=1}^{\infty} b_pz^p  $ be two power series. We define their Dirichlet product as
\[
\sum_{p \geq 1} a_p z^p \otimes \sum_{p \geq 1} b_p z^p = \sum_{p \geq 1} a_p ( \sum_{q \geq 1} b_q z^{qp}) =\sum_{q \geq 1} b_q( \sum_{p \geq 1} a_p z^{pq}) = \sum_{n \geq 1} (a \star b)_n z^n
\]
where $a \star b$ stands for the Dirichlet convolution of the sequences  $(a_n)_{n\geq 1}$  and $(b_n)_{n\geq 1}$. It is clear that the identity element for the binary operation $\otimes $ is $e(z)= z$.
\\
It  should be noted that this product comes from the natural formal product of the two Dirichlet series $\displaystyle \sum_{n=1}^{\infty} \frac{a_n}{z^n}  $ and $\displaystyle \sum_{n=1}^{\infty} \frac{b_n}{z^n}  $. In other words the map
\begin{equation}\label{Spira}
S: \left(H_0^2 (D), +, \otimes \right) \longrightarrow \left({\bf H}, +, . \right), \quad S(\sum_{n=1}^{\infty} a_nz^n)= \sum_{n=1}^{\infty} \frac{a_n}{z^n} 
\end{equation}
is a  ring homomorphism.
\end{definition}
It is possible to define, by transfert of $\otimes $ by the map $S$, a product on the set of Lambert series. But, instead, we give few examples, in particular to show how to compute ${\mathscr M}_0(z)\otimes {\mathscr M}_0(z)$  and evoke the problem that $ {\mathscr C}_{s, l} (z)\otimes {\mathscr C}_{s, l} (z)$         poses.
 \\
First, we have two useful properties 
  \begin{enumerate}
 \item ${\mathscr L}_s(z^m) \otimes {\mathscr M}_s(z^n) = z^{mn},\quad m, n \in \N^*$.
 \item The functions ${\mathscr L}_s(z)$ and ${\mathscr M}_s(z)$  are mutual inverses for the operation $\otimes $. 
 \end{enumerate}
Second,  According to \cite{McCarthy} (p.40) we introduce $d(n, k)$  the number of ways of expressing $n$ as a product of $k$ positive factors (of which any number may be unity), expressions in which the order of the factors is different being regarded as distinct. It  is a multiplicative function satisfying the functional equation
\[ d(n, k+1)= \sum_{d\vert n}  d(n, k). \]
In particular, $d(n, 2)= d(n)= \displaystyle \sum_{d\vert n}1$, the number of divisors of $n$. By a simple induction we have
\[ \smash[b]{ \underbrace{{\mathscr L}_s(z) \otimes\cdots \otimes\,{\mathscr L}_s(z)}_\text{$k$ times}}= \sum_{n=1}^{\infty} \frac{d(n, k)}{n^s}z^n.\]

Third, we  compute
the square of the M\"{o}bius function $\mu$ for the convolution. For any real number $\alpha$, we denote by $\mu_{\alpha}$ the multiplicative function defined for all primes $p $ and positive integer $k$ by \cite{Dickson}
\[ \mu_{\alpha} (1)= 1,\quad   \mu_{\alpha} (p^k)= (-1)^k\binom{\alpha}{k},   \]
with
\[\binom{\alpha}{k}= \frac{\alpha (\alpha-1)\cdots (\alpha -k+1)}{k!}. \]
Then $\mu_1 = \mu$, the M\"{o}bius function, $\mu_{-1} = 1$, the constant arithmetical function $1$, and $\mu_0= e$ with $ e(1)=1, e(n)= 0$ if $n>1$, the neutral element for the Dirichlet convolution. The function $\mu_{\alpha}$ may be defined even for complex $\alpha$ since it is a polynomial in  $\alpha$ \cite{Brown}. It satisfies
\[ \mu_{\alpha + \beta}= \mu_{\alpha}\star \mu_{\beta}\]
for all real numbers $\alpha$ and $ \beta$. Let $n_a$ be the number of simple prime divisors of $n$, that is those primes whose square do not divide 
$n$, then
\[\mu \star \mu(n)=(- 2)^{n_a}. \]
For sake of completeness we give a very quick proof of this result. Since $\mu \star \mu$ is multiplicative, it is enough to know $\mu \star \mu(p^e)$ for a prime $p$. But
\begin{align*}
\mu \star \mu(p^e)&= \sum_{k=0}^e \mu(p^k)\mu(p^{e-k})\\
&=  \mu(p^e)+ \mu(p)\mu(p^{e-1})+\cdots+ \mu(p^{e-1})\mu(p)+\mu(p^e).
\end{align*}
If $e\geq 3$ and $0\leq k\leq e$, one of the integers $k,\,e-k$ is greater than $2$, so $\mu(p^k)\mu(p^{e-k})=0$ and $\mu(p^e)=0$.\\
If $e=2$
\[\mu \star \mu(p^2)=  \mu(p^2)+\mu(p)\mu(p)+\mu(p^2)= \mu(p)\mu(p)=1. \]
If $e= 1, \mu(p)= -1, \mu \star \mu(p)=  \mu(p)+  \mu(p)=-2$. So only the simple prime divisors of $n$ contribute, each by $-2$. This proves the formula above. It follows that
\begin{proposition}
\[  {\mathscr M}_0(z)\otimes {\mathscr M}_0(z)= \sum_{n=1}^{\infty}    (- 2)^{n_a}  z^n .              \]
\end{proposition}
This shows the surprising and not so known  formula
\[ \frac{1}{\zeta^2(s)}=   \sum_{n=1}^{\infty} \frac{(- 2)^{n_a}}{n^s}.\]
The same method gives
\[{\mathscr M}_s(z)\otimes {\mathscr M}_s(z)= \sum_{n=1}^{\infty}    \frac{(- 2)^{n_a}}{n^s}  z^n.         \]
If we set
\[ \smash[b]{ \underbrace{{\mathscr M}_s(z) \otimes\cdots \otimes\,{\mathscr M}_s(z)}_\text{$k$ times}}= \sum_{n=1}^{\infty} \frac{d'(n, k)}{n^s}z^n\]
we get from the equality $\displaystyle  \frac{1}{\zeta^{k+1}(s)}=     \frac{1}{\zeta^{k}(s)}  \frac{1}{\zeta(s) }      $ the relation
\[ d'(n, k+1)= \sum_{d'\vert n}  d'(n, k)\mu(k),\quad d'(n, 2)=  (- 2)^{n_a} = \mu \star \mu(n). \]
As far as we know the map $n\rightarrow d'(n, k)$ does not seem to have been studied to the point like what we have, for example, in the estimate \eqref{Bateman}.
\begin{remark} The situation for the series  $ {\mathscr C}_{s, l} (z) $ is not as manageable as it is for  ${\mathscr L}_s(z)  $  and  $ {\mathscr M}_s(z) $. The
product $ {\mathscr C}_{s, l} (z)\otimes {\mathscr C}_{s, l} (z)$ is not apparently easy to compute, as we have 
 \[c_{q_1}(n)c_{q_2}(n)= c_{q_1 q_2}(n)\]
only when $q_1,q_2$ are relatively prime. We modify the binary operation $\otimes$ by defining for two arithmetical function $f=f(n), g=g(n)$ the
 following operation \cite{McCarthy} (p.154)
\[(f \sqcup g)_n= \sum_{\substack {pq= n\\(p,q)= 1}} f(p)f(q), \]
known as the unitary product,
and extend it to powers series by
 \[\sum_{n\geq 1} f(n)z^n\boxtimes   \sum_{n\geq 1} g(n)z^n= \sum_{n\geq 1} (f\sqcup g)_n z^n.
\]
With $\displaystyle f(n)= \frac{ c_n(l)}{n^s}     $ we get
\[ (f \sqcup f)_n= \sum_{\substack {pq= n\\(p,q)= 1}}f(p) f(q)= {\tilde d}(n)\frac{c_n(s)}{n^s}, \]
where $\displaystyle {\tilde d}(n)= \sum_{(p,q)=1,\, pq= n} 1$, so that
 \[{\mathscr C}_{s, l} (z)\boxtimes {\mathscr C}_{s, l} (z)= \sum_{n \geq 1} c_n(l)  \frac{z^n}{n^s}\boxtimes \sum_{n \geq 1} c_n(l)  \frac{z^n}{n^s} 
= \sum_{n \geq 1} {\tilde d}(n) c_n(l)  \frac{z^n}{n^s}.  \]
The arithmetical function $ {\tilde d}(n)$ is known as the unitary divisor function. It coincide with $d(n)$  if $n$ is square free.
\end{remark} 
To understand the Hilbert space structure of $H_0^2 (D)$ we recall that the Bergman space $B(D)$ is the space of holomorphic functions $f$ in $D$ for which the integral
\[(f,f)= \int \int_{D} \vert f(z)\vert^2 dx\,dy <\infty. \]
The system of functions $\{1, z, z^2,\cdots \}$ is an orthogonal set. Indeed we have
\[(z^n, z^m)= \int \int_{\vert z\vert<1}  z^n {\bar z}^m dx dy= \frac{1}{2i(m+1)} \int_{\vert z\vert=1}z^n {\bar z}^{m+1} dz=  \frac{1}{2(m+1)} \int_{0}^{2\pi}
e^{i(n-m)\theta}\,d\theta. \]
The orthonormalized set is 
\[v_n(z)= \sqrt{\frac{n+1}{\pi}}z^n.\]
The Fourier coefficients of $f\in B(D)$,
\[f(z)=  a_0+a_1z+a_2z^2+\cdots\] are
\[b_n=  \sqrt{\frac{n+1}{\pi}} \int \int_{\vert z\vert<1} f(z) {\bar z}^n dx dy= \lim_{r\to 1}\sqrt{\frac{n+1}{\pi}} \int \int_{\vert z\vert<r} f(z) {\bar z}^n dx dy= 
     \sqrt{\frac{\pi}{n+1}} a_n        \]
     so that  the norm given in the space $  H_0^2 (D)\subset B(D) $ can be written in terms of the Fourier coefficients
     \[\sum_{n=0}^{\infty}\vert a_n\vert^2= \pi  \sum_{n=0}^{\infty} \frac{\vert b_n\vert^2}{n+1}.\]
The following lemma from \cite{Spira}  is the analogue of the classical Cauchy's theorem for the  new binary operation   $\otimes$ 
\begin{lemma}[Spira]
If $\displaystyle f (z) = \sum_ {p \geq 1} a_pz^p $ and $ \displaystyle g(z) = \sum_ {p \geq 1} b_pz^p $ are two holomorphic functions on the open disk $ D $ then so is $ f \otimes g $. Furthermore if $R_f, R_g, R_{f\otimes g}$ are the radius of convergence of $f,g, f\otimes g$ repectively, then
\[\min(1, R_f, R_g)\leq  R_{f\otimes g}.\]
\end{lemma}
\begin{proof} For every fixed $n \geq 3 $  and $2 \leq p \leq n-1 $ we have  $\displaystyle p+ \frac{n}{p} \leq n$ for  each divisor $p$ of $n$. In fact It suffices to show it when
$\displaystyle 2 \leq p \leq \frac{n}{2}$, and in this case $\displaystyle p+ \frac{n}{p} \leq \frac{n}{2}+\frac{n}{2}= n$. For  $\vert z \vert  \leq1$ we have $\displaystyle \vert z \vert^n\leq \vert z \vert^p \vert z\vert^{ \frac{n}{p}} $ and for large $N $:
\begin{equation}
 \begin{split}
\sum_{2 \leq n \leq N} \Big ( \sum_{p\vert n,\; 2 \leq p \leq n-1} \vert a_p \vert \vert b_{ \frac{n}{p}} \vert \Big) \vert z \vert^n & \leq 
\sum_{2 \leq n \leq N} \Big ( \sum_{\substack{p\vert n\\2 \leq p \leq n-1}} \vert a_p \vert \vert b_{ \frac{n}{p}} \vert \Big) \vert z \vert^p \vert z\vert^{ \frac{n}{p}} \\
\\
& \leq \big ( \sum_{2 \leq p \leq N-1} \vert a_p \vert \vert z \vert^p \Big) \Big (\sum_{2 \leq q \leq N-1} \vert b_q \vert \vert z \vert^q\Big)\\
\\
& \leq \big ( \sum_{p \geq 1} \vert a_p \vert \vert z \vert^p \Big) \Big(\sum_{q \geq 1} \vert b_q \vert \vert z \vert^q\Big).
\end{split} \end{equation}
This means that for large $N$:
\begin{equation} \begin{split}
{}&\Big ( \sum_{1 \leq p \leq N} \vert a_p \vert \vert z \vert^p \big) \otimes \Big ( \sum_{1 \leq  q \leq N} \vert b_q \vert \vert z \vert^q \Big)\\ &  \leq \vert a_1 \vert \sum_{q \geq 1} \vert b_q \vert \vert z \vert^q + \vert b_1 \vert \sum_{p \geq 1} \vert a_p \vert \vert z \vert^p  + \Big (\sum_{1 \leq p \leq N}\vert a_p \vert \vert z \vert^p \big ) \Big (\sum_{1 \leq q \leq N}\vert b_q \vert \vert z \vert^q\Big )\\
{} &  \leq \vert a_1 \vert \sum_{q \geq 1} \vert b_q \vert \vert z \vert^q + \vert b_1 \vert \sum_{p \geq 1} \vert a_p \vert \vert z \vert^p  + \Big (\sum_{p \geq 1}\vert a_p \vert \vert z \vert^p \big ) \Big (\sum_{q \geq 1}\vert b_q \vert \vert z \vert^q \Big ).
\end{split}
 \end{equation}
Hence the conclusion.
\end{proof}
 
 \subsection{Historic facts}
\begin{enumerate}
\item Around 1944 A. Wintner \cite{Wintner} shows that for $ u(t)= \{t \} $, the sequence $ u(nt ) $ is total in $ L^2 (0, \frac {1}{2}) $ and observes, for the eventual totality of a sequence $ (\varphi (nt)) _ {n \geq 1}$, the possibility to express the conditions in terms of the M\"obius inversion. He uses the associated Dirichlet series and  shows that  the sequence of dilates $ \varphi _{\tau} (nt) $ for $\displaystyle \varphi_\tau (t)= \sqrt {2} \sum_ {n \geq 1} \frac {\sin nt} {n^\tau} $ with $ \Re \tau> \frac{1}{2} $ is total in $\in  L^2 (0,1). $
\item In 1945 A. Beurling considers  the problem of deciding if the system of the dilates $ (\psi (nt))_{n \geq 1} $  of a function $ \psi \in L^2 (0,1) $ is a total system in $ L^2 (0,1) $. To the development $\displaystyle \sum_ {n \geq 1} a_n \varphi (nt) $ of the function $ \psi $
in the basis $ \varphi_n (t)= \sqrt {2} \sin \pi nt = \varphi (nt) $ with $ \varphi (t)= \sqrt {2} \sin \pi t $, he associates the Dirichlet series $\displaystyle  f(s) = \sum_ {n \geq 1} \frac {a_n} {n^s} $, which converges for $\Re s> \frac{1}{2} $, and studies the initial problem using the  properties of the Dirichlet series $ f $.
\item 
From 1990, very intensive research was carried out to link and exploit the correspondence between the two aspects 
(Fourier series $F \xleftrightarrow{{S}} G$ Dirichlet series) \cite{H.Q}, \cite {HLS},
 \cite{Bayart}, \cite{Nikolski} , \cite{Jaffard}, \cite{Breteche} $\cdots $ and the references therein.
\end{enumerate}

In particular we quote from \cite{HLS} the following
\begin{theorem} Let $\varphi \in L^2(0,1) $ having the following Fourier expansion \[\varphi(t) = \sqrt{2} \sum_{n \geq1} a_n \sin ( \pi nt),\] then the following are equivalent:
\begin{enumerate}
\item The sequence  of dilates $\displaystyle(\varphi_n)_{n \geq 1} $ of $\displaystyle \varphi$ form a Riesz basis of $L^2\left(0,1\right)$.
\item The generating function $\displaystyle  S \varphi(s)= \sum_{n \geq 1} \frac{a_n}{n^s}$ belongs to 
\[\mathcal{ H}^\infty = {\bf H}^\infty (\{ s \in \BC, \Re s >0 \}) \cap  \mathcal{D} \] as well as its reciprocal $\displaystyle \frac{1}{ S \varphi(s)} $, $ \mathcal {D} $ is the ring of convergent Dirichlet series on the half plane $\{ s \in \BC, \Re s >0 \}$. 
\end{enumerate}
In particular the dilates of  $\varphi_\tau$ form a Riesz basis of  $L^2(0,1) $ if and only if $ \Re \tau >1$. In this case $\displaystyle  S \varphi(s)= \zeta(s+ \overline{\tau})= \displaystyle \sum_{ n \geq 1} \frac{1}{n^{s+ \overline{\tau}}} $  and $S^{-1} (s)=  \displaystyle\sum_{n \geq 1} \frac{ \mu(n)}{n^{s+ \overline{\tau}}}$. 
\end{theorem}
Furthermore
\begin{lemma}
The  three rings $ \C([[z]]), \mathcal {O} (D) $ and $ \mathcal {O} _0 $  equipped with the binary composition $ \otimes $ are commutative ring, with neutral element $ z $. The ring of arithmetical functions, equipped with Dirichlet convolution,  is an  integral domain, factorial, local and isomorphic to the ring of Dirichlet series. The same is true of the ring $ \mathcal {D} $. 
\end{lemma}
We recall that a convergent Dirichlet series is a series $\displaystyle \sum_{n=1}^{\infty}\frac{a_n}{n^s}  $ having a finite  abscissa of convergence. This is equivalent to
$a_n= O(n^k) $ for some real positive $k$.
\begin{definition} Let $H$ be a separable Hilbert space. A basis $(x_n)$ is a Riesz basis for $H$ if it is equivalent to some (and therefore every) orthonormal basis $(y_n)$ for $H$,
that is  if there exists a topological isomorphism $L: H \to H$ such that $Lx_n = y_n$ for all $n$.
\end{definition}
The system $\{e^{int},\;n\in \BZ\}$ is a Riesz basis for $L^2[-\pi,\pi] $ and a conditional basis for $L^p[-\pi, \pi] $ with $1<p<\infty,\, p\neq 2$. In evocation of the polylogarithm function,  we cite the following example of  Babenko given in \cite{Babenko}, \cite{Singer} (p.428, Example 14.4): The systems $\{\vert t\vert^{-\vert\beta \vert} e^{int},\;n\in\BZ\}$ 
and $\{\vert t\vert^{\vert\beta \vert}e^{int},\;n\in\BZ\}$  with $0 < \beta < \frac{1}{2}$ are bounded conditional bases for $L^2[-\pi,\pi] $ that are not a Riesz basis.
Naturally all Riesz bases are equivalent, as do all orthonormal bases of  Hilbert spaces. The adjoint mapping $L^*: H\to H$  is a Hilbert space epimorphism.  The sequence  $x_n^* = L^*x_n$ is the biorthogonal sequence of  $(x_n)_{n \geq 1} $. 
We are going to see that the sequence $( {\mathscr L}_\tau(z^n)$ is a Riesz basis of $H^2(D)$ for $\Re\tau >1$, and the corresponding biorthogonal sequence is  $(\psi_n(z))_{n \geq 1}= (\psi_{n, \tau}(z))_{n \geq 1}$ where $\displaystyle  \psi_n(z) = \displaystyle \frac{1}{n^\tau} \sum_{ d \vert n} \mu_{\frac{n}{d}} d^\tau z^d$.\\
We will use the following characterization  of Riesz sequences, due to N.K.Bari \cite{Bari}, \cite{J.B}, \cite{HLS}:
\begin{lemma} \label{Bakri} Let $H$  be a Hilbert separable space and  $B= (x_n)_{n \geq 1}$ be a sequence in $H$. $B$ is a Riesz basis in $H$ if and only if
\begin{enumerate}
\item every $x\in H$ can be expanded as $ \displaystyle x=\sum_n a_n x_n$
\item There exist two constants $0 < c < C < \infty$ such that for every sequence $(a_n)_{n\geq 1}$ with finite support we have:
\[
c \sum_{n \geq 1} \vert a_n \vert^2 \leq \Vert \sum_{n \geq 1} a_n x_n \Vert^2 \leq C \sum_{n \geq 1} \vert a_n \vert^2.
\]
\end{enumerate}
\end{lemma} 
The following lemma uses ideas from  \cite{Wintner}, see also  \cite{J.B}, \cite{HLS}.  
 \begin{lemma} We have the following equalities for $\Re s >1$
 \begin{enumerate}\label{dotproduct} 
 \item 
 $\displaystyle ({\mathscr L}_s(z^m) \vert {\mathscr L}_s(z^n)) = \sum_{\substack{k, l \geq 1\\ km=ln}} \frac{1}{k^sl^s} = \frac{ (\gcd(m,n))^{2s}}{(mn)^s} \zeta(2s)
 $.
 \\
 \item 
 $\displaystyle ({\mathscr M}_s(z^m) \vert {\mathscr L}_s(z^n)) = \sum_{\substack{k, l \geq 1\\ km=ln}} \frac{\mu_k}{k^sl^s} = \frac{ (\gcd(m,n))^{2s}}{(mn)^s} \mu_{\delta(m,n)} \sum_{\substack{j \geq 1\\ (j,\delta(m,n))=1}} \frac{ \mu_j}{j^s}.$
 \end{enumerate}
 with, if $\displaystyle f(z)= \sum_{n=0}^{\infty}a_nz^n,\;    g(z)= \sum_{n=0}^{\infty}b_nz^n,           $
 \[(f(z)\vert g(z))=  \sum_{n=0}^{\infty}a_n{\bar b}_n.\]
 \end{lemma}
 \begin{remark}
  The basic example is provided by the Hilbert space $L^2(0, \pi)$ and the dilates $(u_n)$ of the function $\displaystyle u(x) = \sum_{ k \geq 1} \frac{\sin k x}{k^s} $, with
\[
 (u_m \vert u_n)=\frac{ \pi}{2} \sum_{\substack{ k, l \geq 1\\ km= l n}} \frac{1}{k^s l^s} =\frac{\pi}{2} \zeta(2s) \frac{(\gcd(m,n))^s}{ (mn)^s}.
 \]
 \end{remark}
 \begin{lemma} Let $(c_n)_{n \geq 1} $ a sequence with finite support of complex numbers. Then
 \begin{enumerate}
 \item
 $ \Vert \displaystyle \sum_{ n \geq 1} c_n {\mathscr L}_s(z^n) \Vert^2 =\displaystyle \sum_{m, n \geq 1} c_m \overline{c}_n ({\mathscr L}_s(z^m) \vert {\mathscr L}_s(z^n)) = \zeta(2s) \sum_{m,n \geq 1} \frac{(\gcd(m,n))^{2s}}{(mn)^s}  c_m \overline{c}_n.$\\
 \item  \begin{equation*}\begin{split}\Vert \displaystyle \sum_{ n \geq 1} c_n {\mathscr M}_s(z^n) \Vert^2 &= \displaystyle \sum_{m,n \geq 1} c_m \overline{c}_n ({\mathscr L}_s(z^m) \vert {\mathscr L}_s(z^n))\\
 & = \displaystyle \sum_{m,n \geq 1} \frac{(\gcd(m,n))^{2s}}{(mn)^s}  c_m \overline{c}_n \mu_{  \frac{n}{\lcm(m,n)} }\displaystyle \sum_{\substack{ j \geq 1\\ (j,\, \delta(n,m))  
 = 1}} \frac{\mu_j}{j^{2s}}.\\  \end{split}\end{equation*}
  \end{enumerate}
 \end{lemma}
 \begin{remark} We thus see appearing the  $N \times N $ symmetric square matrices
 \[
 M_{s,N} = \left ( \frac{ (\gcd(m,n))^{2s})}{(mn)^s }\right)_{ 1\leq m,n \leq N}.
 \]
 It is possible to compute the determinant of the matrix $M_{s,N}$. We recall first that the Smith determinant is defined to be 
 \[\Delta_N=  {\rm det} \left(\gcd(m,n)\right)_{1\leq m,\, n\leq N}\]
  and its value, in terms of the Euler's totient function $\Phi$, is \cite{Smith} 
  \[ \Delta_N= \Phi(1)\Phi(2)\cdots \Phi(N).\]
 The determinant $\Delta_N^{(r)}=  {\rm det} \left(\gcd(m,n)^r\right)_{1\leq m,\, n\leq N} $  where $r$ is a real number, was also evaluated by Smith in \cite{Smith}. To explain the value of $\Delta_N^{(r)}$ we introduce  the Jordan's totient function  $ {\displaystyle J_{k}}$ given by  \cite{Smith}, \cite{Andrica},
  \cite{Thajoddin}
\[ {\displaystyle J_{k}(n)=n^{k}\prod _{p|n}\left(1-{\frac {1}{p^{k}}}\right)\,},\] where  ${\displaystyle p} $ ranges through the prime divisors of  ${\displaystyle n}$ . We have $ \displaystyle J_{1}(n)= \Phi(n)$. Furthermore 
   \[{\displaystyle \sum _{d|n}J_{k}(d)=n^{k}.}\]
which may be written as convolution product as
\[\displaystyle J_{k}(n)\star 1=n^{k}\,\]
and by a M\"{o}bius inversion 
\[  \displaystyle J_{k}(n)= \mu (n)\star n^{k}.\]
The Dirichlet generating function of  series for  $ {\displaystyle J_{k}} $ is
\[ \displaystyle \sum _{n\geq 1}{\frac {J_{k}(n)}{n^{s}}}={\frac {\zeta (s-k)}{\zeta (s)}}.\]
 Similarly to the case of $ \Delta= \Delta^{(1)} $,  we have the formula
\[\Delta_N^{(r)}=  J_r(1) J_r(2)\cdots  J_r(N).  \]  
Since the determinant is a multilinear form we obtain
\[{\rm det }\, M_{s,N}= \frac{1}{(N!)^{2s}}  J_{2s}(1) J_{2s}(2)\cdots  J_{2s}(N),\]
so the matrices $ M_{s,N}$ are invertible. These statements remain valid for every $s\in \C$ by analytic continuation. The matrices $ M_{s,N}$ are also positive for $ s\in (1, + \infty )$ and according to \cite{Wintner} and \cite{LINDSEIP},  the smallest eigenvalue of $ \lambda_N (s) $ and the largest eigenvalue  $ \Lambda_N (s) $  of $ M_ {s, M} $ satisfy
 \begin{equation}\label{eigenvalues}
\frac{ \zeta(2s)}{\zeta(s)^2} \leq \lambda_N(s) \leq \Lambda_N(s) \leq \frac{ \zeta(s)^2}{ \zeta(2s)}.
 \end{equation}
 We deduce that for a sequence  $a= (a_n)_{ n \geq 1} \in l^2$ we have  \cite{LINDSEIP}
 \[
  \frac{ \zeta(2s)}{(\zeta(s))^2} \sum_{1 \leq n \leq N} \vert a_n \vert^2 \leq \sum_{1 \leq m,n \leq N} \frac{(\gcd(m,n))^{2s}}{(mn)^s} a_m \overline{a}_n  \leq \frac{ (\zeta(s))^2}{ \zeta(2s)} \sum_{1 \leq n \leq N} \vert a_n \vert^2.
 \]
 \end{remark}
 This prove the following 
 \begin{proposition}
 If $ s >1, $ the sequence $({\mathscr L}_s(z^n))_{n \geq 1} $ is a Riesz basis of $H^2(D)$.
 \end{proposition}
  By direct computation we see that the associated biorthogonal basis is $(\psi_n(z))_{n \geq 1}$ where
 \begin{equation}\label{ortho}
 \psi_n(z)= \psi_{n,s}(z)= \frac{1}{n^s} \sum_{d \vert n} \mu_{\frac{n}{d}} d^s z^d.
 \end{equation}
 The two extreme factors in \eqref{eigenvalues} have interesting Dirichlet series expansion \cite{McCarthy} (p.227)
 \begin{equation*}
 \frac{\zeta^2(s)}{\zeta(2s)}= \sum_{n=1}^{\infty} \frac{\theta(n)}{n^s},\quad \frac{\zeta(2s)}{\zeta^2(s)}=  \sum_{n=1}^{\infty} \frac{\lambda(n) \theta(n)}{n^s}.
 \end{equation*}
 The function
 $\theta(n)$ is defined by
 \[ \theta(n)= 2^{\omega(n)},\]
 where $\omega(n) $ is the number of different prime factors of $n$. It is a multiplicative function, also related to the M\"{o}bius function by $\displaystyle \theta(n)= \sum_{d\vert n} \vert \mu(d)\vert $.
 \begin{remark}
 It is worth noting that the positivity of the matrix  $ M_{s,N}$  can be deduced from the Franel integral \eqref{Chowla3}. If for a suitable real function  $f$ we have  two real $s, s'$ such that for every  $1\leq m,n\leq N$
 \[\int_0^1 f(mx)f(nx) dx= \frac{\gcd(m,n)^s}{m^{s'}n^{s'}}\]
 then for $c_1, \cdots,c_N \in \C$
 \[ \sum_{1\leq m,n\leq N}   \frac{\gcd(m,n)^s}{m^{s'}n^{s'}}c_m\bar{c}_n=  {\bigintss_0^1} \left \vert  \sum_{p=1}^{p=N}f(px) c_p\right \vert^2\,dx\geq 0\]
and \[ \sum_{1\leq m,n\leq N}  \gcd(m,n)^s c_m\bar{c}_n=     {\bigintss_0^1}  \left \vert  \sum_{p=1}^{p=N}p^{s'} f(px) c_p\right \vert^2\,dx\geq 0.   \]
We have the definite positivity if the functions $x \to f(px), 1\leq p\leq N  $ are linearly independent.
 \end{remark}
  \subsection{Multipliers }
 We consider new spaces of Dirichlet series.
 \begin{enumerate}
 \item The space $\mathcal{ H}^2 =\left\{ \displaystyle \sum_{n \geq 1} \frac{a_n}{n^s}, \,a=(a_n)_{n \in \N^*} \in l^2 \right\}$ corresponding to the spaces $L^2(0,1)$ et $H^2(D)$.
 \item The space $ \mathcal{ H}^\infty =  {\bf H}^\infty (\{ s \in \C, \Re s >0\}) \cap \mathcal{ D}$ equipped with the usual norm $\Vert \; \Vert_\infty$, defined on the space of  measurable and bounded functions defined on  $\{ s \in \C : \Re s >0 \}$.
 \end{enumerate}
  It is easily shown that $ \mathcal{ H}^\infty  \subset \mathcal{ H}^2 $ and that $\Vert f \Vert^2 \leq \Vert f \Vert_\infty $  for $ f \in \mathcal{ H}^\infty$. The set of multipliers of $\mathcal{ H}^2 $ can be identified with $\mathcal{ H}^\infty$. The norm of the multiplier, as operator,  
 $M_\varphi : \mathcal{ H}^2 \ni f \longrightarrow \varphi f \in \mathcal{ H}^2$ is  $\Vert M_\varphi \Vert = \Vert \varphi \Vert_\infty$. We also have the following interesting  property: Let $\varphi \in \mathcal{ H}^\infty$,  the multiplier  operator   $M_\varphi $ is an isomorphism of  $\mathcal{ H}^2$ if and only if $\varphi ^{-1} \in \mathcal{ H}^\infty$. In this case $\Vert M_\varphi^{-1} \Vert = \Vert \varphi^{-1} \Vert_\infty$. 
  In the correspondence between  power series  and  Dirichlet series  the multiplier set of $ H_0^2(D) $ for $ \otimes $ is identified with the set of power series
 $ \displaystyle \varphi (s) = \sum_ {n \geq 1} {\alpha_n} z^n$  such that the function $ \displaystyle \varphi (s) = \sum_ {n \geq 1} \frac{\alpha_n} {n^s} $ belongs to $ \mathcal {H}^{\infty }$.  An illustration of this fact is given by the polylogarithm function: 
 For $\Re \tau >1 $ the series ${\mathscr L}_{\tau}(z) $ and ${\mathscr M}_\tau(z)$ are reciprocal multipliers of $H_0^2(D)$ for the the operation $\otimes$.
 Moreover the image of the multiplier $  {\mathscr L}_s(z)$ by the map $S$ is the translate of the zeta function $\displaystyle \zeta(s+\tau)= \sum_{n\geq 1} \frac{1}{n^{s+\tau}}$.\\
\begin{example}We now give  an example of the expansion of a given  $g \in H_0^2(D)$ in the  Riesz basis $({\mathscr L}_s(z^n))_{n \geq 1}$. We need to find a sequence $(\alpha_n)_{n\geq 1}$ such that $\displaystyle g(z) = \sum_{k \geq 1} \alpha_k {\mathscr L}_s(z^k)$. If $f(z)$ is the formal power series $ f(z) = \sum_{k \geq1} \alpha_k z^k $ , then $g(z)$ should be $g(z)= f(z)\otimes {\mathscr L}_s(z)$ or $ f(z)= g(z) \otimes {\mathscr M}_s(z)$. According to the
Lemma \eqref{Spira} the convergence radius $R_f$  of $f$ satisfies $\displaystyle 1= \min(1, R_{{\mathscr M}_s}, R_g)\leq  R_f$ and 
\[f(z)= \sum_{n \geq 1} \big( \sum_{d \vert n} \frac{ \mu_{\frac{n}{d}}}{ (\frac{n}{d})^s} a_d \big) z^n = \sum_{n \geq 1}\frac{1}{n^s} \big( \sum_{d \vert n}  \mu_{\frac{n}{d}} d^s a_d \big ) z^n.\] We thus see that in terms of the biorthogonal basis \eqref{ortho}, $ \alpha_n = (g \vert \psi_n)$, naturally enough.
\end{example}
\subsection{On the Estermann's function}
The Estermann zeta function $E(s,a,z) $ is defined by the Dirichlet series
\begin{equation}\label{Estermann}
E(s,a,z) = \sum_{n \geq 1} \frac{\sigma_a(n)}{n^s} z^n  \quad {\Re}s > 1+\Re a,\; \vert z \vert \leq 1,
\end{equation}
where, as already denoted, $\displaystyle   \sigma_a(n) = \sum_{d\vert n} d^a, \, a\in \BC $. This Dirichlet series is closely related to Ramanujan sums. 
This series \eqref{Estermann} can be given in terms of the polylogarithm function, or more precisely can be expanded with respect to the Riesz basis
$\{{\mathscr L}_s(z^n),\, n\geq 1\} $. In fact if  $a>0$ then
\[
E(s,a,z) = \sum_{p \geq 1} \frac{1}{p^{s-a}} {\mathscr L}_s(z^p), \quad \vert z \vert \leq 1, \; \Re s > \mathrm{max} (1, 1+a ).
\]
This can be shown either by observing  that          \[(E \vert \psi_p)= \frac{1}{p^s} (\mu \star \sigma_a)(p) = \frac{1}{p^{s-a}} \] by using  $\sigma_a = I \star u^a$  with $I(n)=1 $ (the constant arithmetical function) and  $u^a(n)=n^a $ for every  $n \in \N^*$, or by showing  that
\[E(s,a, \bullet) \otimes {\mathscr M}_s= {\mathscr L}_{s-a}.\]
Assume that $a<0$ and  recall first that for $\Re(s+ \tau) >1$ then \eqref{ramanujan}
\[
\zeta(s+ \tau) \sum_{d \vert k} \mu_{\frac{k}{d}} d^{1-s-\tau} = \sum_{ n \geq 1} \frac{ c_k(n)}{n^{s+\tau}}.
\]
We deduce that for $t>0$ we have
\begin{equation*} \begin{split}
\sum_{n \geq 1} \frac{c_k(n)}{n ^\tau} e^{-nt} & = \frac{1}{2i \pi} \int_{c- i \infty}^{c+ i \infty} \Gamma(s) \sum_{n \geq 1} \frac{ c_k(n)}{n^{s+ \tau}} \frac{ds}{t^s} \\
&= \frac{1}{2i \pi} \int_{c- i \infty}^{c+ i \infty} \Gamma(s) \sum_{n \geq 1} \zeta(s+\tau) \sum_{d \vert k} \mu_{\frac{k}{d}} d^{1- \tau} \frac{ds}{(dt)^s} \\
&= \sum_{d \vert k} \mu_{\frac{k}{d}} d^{1- \tau} \sum_{n \geq 1} \frac{e^{-ndt}}{n^\tau}.\\
\end{split} \end{equation*}
We obtain, by analytic continuation, that
\[
\sum_{ n \geq 1} \frac{c_k(n)}{n^\tau} z^n = \sum_{d \vert k} \mu_{ \frac{k}{d}} d^{1- \tau} {\mathscr L}_{\tau} (z^d).
\]
 Since  $a<0$, we get from \eqref{Ramanujan2}  \[\sigma_a(n) = \zeta(1-a) \sum_{k \geq 1} \frac{ c_k(n)}{k^{1-a}}.\] Thus
\begin{equation*} \begin{split}
E(s,a,z) & = \sum_{ n \geq 1} \frac{\sigma_a(n)}{n^s} z^n \\
&= \zeta(1-a) \sum_{k \geq 1} \frac{1}{k^{1-a}} \left ( \sum_{n \geq 1} \frac{ c_k(n)}{n^s} z^n \right )\\ 
&= \zeta(1-a) \sum_{k \geq 1} \frac{1}{k^{1-a}} \left ( \sum_{d \vert k} \mu_{ \frac{k}{d}} d^{1-s} {\mathscr L}_s (z^d) \right).\\
\end{split} \end{equation*}
 \section{A link with Kubertt identities} 
 Let $\displaystyle f(z )= \sum_{ n \geq 1} a_n z^n $ be a power series, convergent for $\vert z \vert \leq 1$.  The  condition \[\sum_{ 1 \leq h \leq m} f( e^{ 2i \pi \frac{h}{m}}) =0 \] is equivalent to
 \[
 \sum_{ n \geq 1} a_n \left( \sum_{ 1 \leq h \leq m}  e^{ 2i \pi \frac{nh}{m}}\right) = m \sum_{j \geq 1} a_{jm} =0.
 \]
Furthermore  if we assume that  \[\sum_{1 \leq h \leq m} f(e^{2i \pi \frac{h}{m}}) =0\] for every $m \in \N^* $, then for every $m \in \N^*$
  \[\sum_{j \geq 1} a_{jm}= 0.\]
  The function       ${\mathscr M}_1$ satisfies this property.  Furthermore  according to \cite{H.Q} (p.294) if $a=(a_n)_{n \in \N^*} $ is a non zero sequence
   satisfying this property, then  $a \notin \ell^1$, the space of sequences whose series is absolutely convergent. This last condition is also satisfied by ${\mathscr M}_1 $. In fact if $\vert \mu(n)\vert= \mu^2(n) $ is the characteristic function of squarefree integers, then \cite{McCarthy} (p.227)
   \[\sum_{n=1}^{\infty}  \frac{\vert \mu(n)\vert}{n^s}=  \frac{\zeta(s)}{\zeta(2s)}\]
  and $\displaystyle \lim_{s\to 1} \sum_{n=1}^{\infty}  \frac{\vert \mu(n)\vert}{n^s}= \infty       $. Obviously this also results from \eqref{subseries3}. \\
   One wonders for what  $f \in \mathcal{ O }_0 $ the product $f \otimes {\mathscr M}_1$ will verify Kubert identity. Naturally, quite strong convergence hypothesis on the sequence  $a=(a_k)_{k\geq 1}$ will be required. We recall that  
 \[f\in H^2(D)  \rightarrow f \otimes {\mathscr M}_\tau \in H^2(D) \]
   is an  isomorphism. If $f(z) = \sum_{k \geq 1} a_k z^k $ with suitable convergence on  the circle ${\bf U}= \{\vert z \vert =1\} $ we can write
 \[
 \sum_{1 \leq h \leq m} (f \otimes {\mathscr M}_\tau) (e^{ 2 i \pi \frac{h}{m} }) = \sum_{k \geq 1} a_k \left ( \sum_{ 1 \leq h \leq m} {\mathscr M}_\tau( e^{2 i \pi \frac{h}{m}}) \right ) = \frac{1}{ m^{\tau-1}} \sum_{ k \geq 1}a_k \sum_{ j \geq 1} \frac{ \mu_{jD(m,k)}}{j^\tau}.
 \]
We need two lemmas, the first is strongly inspired by \cite{PP}.
\begin{lemma} Let $D \in \N^*$. For every $1 < \tau <\frac{3}{2}$ we have
$$
\left \vert  \sum_{ j \geq 1}  \frac{ \mu_{jD}}{j^\tau} \right \vert  =   \left \vert \frac{\mu(D)}{ \zeta( \tau)} \prod_{p \vert D} (1- p^{-\tau})  \right \vert \leq  e \frac{ \tau-1}{\zeta(\tau)}.
$$
\end{lemma} 
\begin{proof} 
We introduce
\[
P_D (\tau) = \prod_{p \vert D} (1- p^{-\tau}),\; \quad \tau >1
\]
so we have
\[
\log P_D(\tau) =- \sum_{p \vert D} \log(1-p^{-\tau})=  \sum_{p \vert D} \frac{1}{p^ \tau} + \sum_{p \vert D} \sum_{k \geq 2} \frac{1}{k p^{k\tau}}.
\]
Now
 \begin{align*} 
 \sum_{p \vert D} \sum_{k \geq 2} \frac{1}{k p^{ k \tau}} & \leq \sum_{p\vert D} \sum_{k \geq 2} \frac{1}{k p^{ k \tau}} \leq \frac{1}{2} \sum_{p\vert D} \sum_{k \geq 2} \frac{1} {p^{k \tau}} \\
 & \leq \frac{1}{2} \sum_{p\vert D} \sum_{k \geq 2} \frac{1}{p^k} = \frac{1}{2} \sum_{p\vert D} \frac{1}{p(p-1)}\leq  \frac{1}{2}.
 \end{align*}
 And
 \[
\log P_D( \tau) = - \sum_{p \vert D} \log(1- p^{-\tau}) \leq \sum_{p\vert D} \frac{1}{p^{\tau}} + \frac{1}{2}.
 \]
 The same ideas give
\[
0 < \log \zeta(\tau) - \sum_{p} \frac{1}{p^\tau} < \frac{1}{2} \quad \tau >1.
\]
On the other hand it is easily seen that
\[
1 < (\tau-1) \zeta(\tau) < \tau, \quad \tau >1
\]
and for  $\tau < \frac{3}{2}$,
\[
0 < \log \zeta(\tau)+ \log( \tau-1) < \log \tau < \log \frac{3}{2}.
\]
By putting together these results we arrive to
\[
-\frac{1}{2} < \sum_{p} \frac{1}{p^\tau}+ \log(\tau-1) < \log\frac{3}{2} < \frac{1}{2}.
\]
Hence, for $\tau \in ]1, \frac{3}{2} [$, we have
 \[
\vert  \sum_{p}\frac{1}{p^\tau}+ \log (\tau-1) \vert \leq \frac{1}{2}
\]
or
\[
\sum_{p} \frac{1}{p^\tau} \leq \frac{1}{2}+ \log(\tau-1), \quad 1 < \tau < \frac{3}{2}.
\]
Finally, using $\log P_D( \tau)  \leq  \displaystyle \sum_{p} \frac{1}{p^s} + \frac{1}{2}$, we find that
\[
\log P_D( \tau)  \leq 1+ \log(\tau-1)
\]
or
\[
P_D(\tau ) \leq  e (\tau-1), \quad 1 < \tau < \frac{3}{2}.
\]
Hence the lemma.
\end{proof}
Before to state the following lemma we would like to make a comparative remark. Let 
${\bf T}= \R/\Z$ be the circle, and let  $f: {\bf T}\to \C$ be an integrable function. If the Fourier coefficients $\hat{f}(n)$ satisfy
\[ \sum_{n\in \Z} \vert \hat{f}(n) \vert^2< \infty,\]
then Carleson theorem \cite{Carleson} asserts that the sequence $\displaystyle (S_n(x)_{n\geq 0},\; S_n(x)= \sum_{\vert k\vert\leq n} \hat{f}(k)e^{2i\pi kx}  $
converges to $f(x)$ for almost all $x\in {\bf T}$. In particular  for every $\tau$ with $ \Re\tau > \frac{1}{2},\; \displaystyle \sum_{n \geq 1} \frac{ \mu(n) z^n}{n^\tau} $ converges for almost every  $ z \in {\bf T}$ . The next lemma is a form of  Jensen lemma adapted to the Dirichlet series
$\displaystyle \sum_{n \geq 1} \frac{ \mu(n) z^n}{n^\tau}, \; \vert z \vert \leq 1 $ and $\Re\tau > 1$. It differs considerably from what we got from Carleson's theorem.
\begin{lemma}
For every $\tau,\,\Re\tau > 1$, the  series $\displaystyle \sum_{n \geq 1} \frac{ \mu(n)}{n^\tau}z^n $  converges uniformly on the angle
\[\vert z \vert \leq 1,\: \vert \arg(z-1)  \vert < \frac{\pi}{2} - \delta, \]
 where $\delta \in (0, \frac{\pi}{2})$ is fixed
\end{lemma}
\begin{proof} Let $\displaystyle a_n(z) =  \mu(n)\,z^n $ be the coefficients of this Dirichlet series. By using  the fundamental estimate of Davenport 
\eqref{Bateman} we can show that  $\displaystyle b_p(z)= \sum_{n \geq p} a_n(z)$ converges uniformly to zero on the unit circle, hence on the closed unit disc, by the maximum principle. For every $ \varepsilon >0 $ there exists $p_\varepsilon \in \N^*$ such that $p \geq p_\varepsilon $ we have $\vert b_p (z) \vert < \varepsilon$. Let \[\displaystyle S_N=S_N(z,s) = \sum_{1 \leq n \leq N} \frac{a_n(z)}{n^s},\; \vert z \vert \leq 1,\;\Re s >0, \: \vert \arg s \vert \leq \frac{ \pi}{2}- \delta.\]  For $Q > P \geq p_\varepsilon$, we have by  partial summation 
\[
S_Q-S_{P-1} = \frac{b_Q}{Q^s} + b_{Q-1} \Big ( \frac{1}{ (Q-1)^s}-    \frac{1}{ Q^s}  \Big ) + \cdots + b_P  \Big ( \frac{1}{ P^s}-   \frac{1}{ (P+1)^s}  \Big )- \frac{b_{P-1}}{P^s}.
\]
and
\[
\vert S_Q - S_{P-1} \vert \leq \varepsilon \left ( \frac{1}{ \vert Q^s \vert } +\frac{1}{ \vert P^s \vert } + \left \vert \frac{1}{ (Q-1)^s}- \frac{1}{ \vert Q^s \vert } \right \vert = \cdots + \left \vert \frac{1}{ P^s}- \frac{1}{ \vert (P+1)^s \vert } \right \vert \right ).
\]
Now
\[
 \frac{1}{(k+1)^s }- \frac{1}{k^s} = s \int_{ \log k}^{\log(k+1)} e^{ - \lambda s} d \lambda,
 \]
 so, with $\sigma = \Re s$
 \[
 \left \vert  \frac{1}{(k+1)^s }- \frac{1}{k^s}  \right \vert \leq  \frac{ \vert s \vert}{ \sigma} \Big ( \frac{1}{k^\sigma }- \frac{1}{(k+1)^\sigma} \Big )
 \]
 and
 \[
  \vert S_Q-S_{P-1} \vert \leq 2 \varepsilon (1+ \frac{1}{ \sin \delta}),
  \]
  hence the lemma.
  \end{proof}
  As a consequence  we find that $\displaystyle  \lim_{\tau \to 1, \tau> 1} {\mathscr M}_\tau (z) = {\mathscr M}_1(z) $ uniformly with respect to $ z $ in the closed unit disk. So, if the sequence $ a = (a_k)_{k \geq 1} $ is reasonable  we will obtain that the function $ f \otimes {\mathscr M}_1 $, with $ f(z)= \sum_ {k \geq 1} a_k z^k $, satisfies the property of Besicovich. For example, if $ a \in l^1 $ (which also ensures that $ f \in H^2(D) $).
  
 \begin{theorem} If $f(z)= \sum_{k \geq 1} a_k z^k$ with $a= (a_k)_{k \geq 1}\in l^1$, then
  \[
  \sum_{1 \leq h \leq m} (f \otimes {\mathscr M}_1)(e^{2 i \pi \frac{h}{m}}) =0.
  \]
  \end{theorem}
  \begin{proof}
  The two lemmas above ensure the possibility of passing to the limit  $ \tau \to 1, \; 1 <\tau <\frac{3}{2} $ in the following relation:
   \[
 \sum_{1 \leq h \leq m} (f \otimes {\mathscr M}_\tau) (e^{ 2 i \pi \frac{h}{m} }) = \sum_{k \geq 1} a_k \left ( \sum_{ 1 \leq h \leq m} {\mathscr M}_\tau( e^{2 i \pi \frac{h}{m}}) \right ) = \frac{1}{ m^{\tau-1}} \sum_{ k \geq 1}a_k \sum_{ j \geq 1} \frac{ \mu_{jD(m,k)}}{j^\tau}.
 \]
 \end{proof}
 If $\displaystyle \sum_{n \geq1} a_n z^n \in H^2(D)$ is orthogonal to the subspace generated by the family of
  \[
 \sum_{ 1 \leq \nu \leq N} c_\nu  \left ( \sum_{ l \geq 1} (-1)^l \frac{\zeta(2 l)}{2 l!} (n \pi \theta_\nu)^{2 l} \right ),
 \]
 that is
 \[
 \sum_{n \geq 1} \overline{a_n} \left[ 
 \sum_{ 1 \leq \nu \leq N} c_\nu  \left (   \sum_{ l \geq 1} (-1)^l \frac{\zeta(2 l)}{2 l!} (n \pi \theta_\nu)^{2 l} \right ) \right ] =0.
 \]
  \begin{remark}
Let $f(z) = \sum_{k \geq 2} a_k z^k$ be a power series of radius at least equal to $R >0$. For $\vert z \vert <R $ and $\Re s > \frac{1}{2}$ we have \cite{CHP}, 
\cite{G.J}
\[
\sum_{ n \geq 1} f(\frac{z}{n^s}) = \sum_{ k \geq 2} a_k \zeta(k s) z^k.
\]
\begin{proof} For $z \in D(0,R) $ we choose $\varepsilon >0 $ and $\alpha >0$ such that $\vert z \vert < R-\alpha < R-\varepsilon <R$. There is  $k_0 \in \N^* $ such that, for $k > k_0$  we have $\displaystyle \vert a_k \vert \vert z \vert ^k< \frac{(R-\alpha)^k}{(R-\varepsilon)^k}$, which ensures the convergence of the series $\sum_{k \geq 2} \vert a_k \vert \vert z \vert^k$. Since  $\Re (k s) >1 $ for $\Re s > \frac{1}{2}$ and $k \geq 2$, the series  $\displaystyle \sum_{ k \geq 2} \vert a_k \vert \vert z \vert^k \sum_{n \geq 1} \frac{1}{n^{ks}} $ is summable. We can therefore apply Fubini's theorem to exchange the summations.
\end{proof}
As an example we take $\displaystyle a_{2n} =(-1)^n \frac{1}{2n!} $ and $a_{2n+1}=0 $ for every $n \geq 1$. This gives
 \[\displaystyle f(z)= \sum_{ k \geq 2} a_k z^k = \sum_{k \geq 1} \frac{ (-1)^k }{2k!} z^{2k} = \cos z -1 \] which converges for every $z\in \C$ and insure, with $s=1$,
\begin{equation}\label{Flett}
 \sum_{j\geq 1} \big ( \cos(\frac{z}{j})-1\big) = \sum_{k \geq 1} (-1)^k \frac{ \zeta(2k)}{2k!} z^{2k}.
 \end{equation}
 We investigated some functions related to the right side of \eqref{Flett} in \cite{SebGay}.
 \end{remark} 
 \section{Asymptotic expansion}
In this section we give the asymptotic expansion of the coefficients of the power series in $\displaystyle H_0^2(D)$ corresponding to
 $\displaystyle \sum_{1 \leq \nu \leq N} c_\nu \{\frac{ \theta_\nu}{x}\},\,\sum_{1 \leq \nu \leq N} c_\nu \theta_\nu = 0$ as element of $L^2(0, 1)$. This is given by the following result

\begin{theorem} The $n$-th coefficient $a_n$ of the power series belonging to $H^2(D)$ and corresponding to $\displaystyle \sum_{1 \leq \nu \leq N} c_\nu \{\frac{ \theta_\nu}{x}\} $ with $\displaystyle \sum_{1 \leq \nu \leq N} c_\nu \theta_\nu = 0$  is given by
\[
a_n =\frac{\sqrt 2}{\pi n} \Big(\displaystyle \sum_{1 \leq \nu \leq N} c_\nu (n \pi \theta_\nu)^{\sigma_0} \Big) o_{\sigma_0}(n)
\]
with $\sigma_0$  such that $\frac{2}{3} < \sigma_0  <1$ and \[o_{\sigma_0}(n) = \frac{1}{2 i \pi} \int_{-\infty}^{+ \infty} \Gamma(-(\sigma_0+i \tau)) \zeta(\sigma_0+ i \tau) \cos( \frac{ \pi}{2} (\sigma_0 +i \tau)) n^{i \tau} d \tau
\] tending to zero as $n$ tends to $ + \infty$.
\end{theorem}
\begin{proof} 
We  look for an asymptotic expansion of the function $\displaystyle  f (x) = \sum_{l \geq 1} (-1)^ l \frac {\zeta (2 l)} {2 l! }x^{2 l} $ when $ x$ tends to  $ +\infty $. For this we consider the line integral 
\[ I_ {M, T} (x) = \frac{1}{2i\pi} \int_ {\gamma_ {M, T}} \Gamma (-s) \zeta (s) \cos (\frac {\pi}{2} s) x^s ds ,\] 
where $\gamma_ {M, T} $ is the rectangular circuit which sides are parallel to the axis, and  whose vertices are the points $ \sigma_0 \pm iT $ and $ M+\frac{1}{2} \pm i T $ for $ M \in \N^{*},\, T>0 $. We recall that
\begin{enumerate}
\item For every fixed $\sigma \in \R $ there exists $C_\sigma > 0 $ such that $\displaystyle \vert \Gamma( \sigma+ i \tau) \vert \leq C_\sigma \vert \tau \vert^{ \sigma - \frac{1}{2}} e^{- \frac{\pi}{2} \vert \tau \vert}$.
\item For $0 \leq \sigma \leq 1$ and $\varepsilon >0 $, there exists $D_\varepsilon >0 $ such that $\displaystyle \vert \zeta( \sigma+ i \tau) \vert \leq D_\varepsilon \vert \tau \vert^{\frac{1-\sigma}{2} + \varepsilon}$.
\item We have $\displaystyle \vert \cos \frac{ \pi}{2} (\sigma+ i \tau) \vert \leq  e^{ \frac{\pi}{2} \vert \tau \vert }$.
\end{enumerate}
Hence for  $s = \sigma_0+i \tau$ with $\sigma_0 \in (\displaystyle \frac{2}{3},1) $ we have the estimate
\[
\left \vert \Gamma(-s)  \zeta(s) \cos( \frac{\pi}{2} s) x^s \right \vert \leq C_{-\sigma_0} C_\varepsilon \frac{x^{\sigma_0}}{ \vert \tau \vert^{ \sigma_0+ \frac{ \sigma_0}{2} -\varepsilon}},
\]
with $\displaystyle \sigma_0+ \frac{ \sigma_0}{2} - \varepsilon >1$ for $\varepsilon >0 $ small enough, since $ \sigma_0+ \displaystyle  \frac{ \sigma_0}{2} >\displaystyle \frac{2}{3}+\frac{1}{3} =1$. 
\subsection{The  integral on the line $\Re s = \sigma_0$} We recall that the function 
\[\R \ni \tau \to  \Gamma(-(\sigma_0+i \tau)) \zeta(\sigma_0+ i \tau) \cos( \frac{ \pi}{2} (\sigma_0 +i \tau)) x^{\sigma_0} \in L^1(\R).\]
\subsection{Integrals on  $ \Re s = M + \displaystyle \frac{1}{2}$  and on horizontal lines}
For every  $s \in \C $ and $M \in \N^*$ such that $\Re s = M+\displaystyle  \frac{1}{2}$ we have $ s- (-\overline{s} ) = 2 \Re s =2M+1$ and
\[
 \Gamma(-s)  = \frac{ \Gamma( -s+2M+1)}{(-s)(1-s) \ldots (2M-s)} = \frac{\overline{ \Gamma(s)}}{ (-s)(1-s) \ldots (2M-s)}.
 \]
 According to the Stirling formula, for $ 0 \leq  \vert  {\arg} s \vert  \leq \frac{ \pi}{2}$
 \[
 \vert \Gamma(s) \vert \leq  C \vert s \vert^{\sigma-\frac{1}{2}}e^{-\sigma} e^{ -\frac{ \pi}{2} \vert \tau \vert }.
 \]
 Then, with $\Re s= \sigma = M+ \frac{1}{2}$, we have $\vert \Gamma(s) \vert \leq C \vert s \vert^M e^{ -(M+\frac{1}{2})} e^{ -\frac{ \pi}{2} \vert \tau \vert }$. Hence
 \[
 \vert \Gamma(-s) \vert \leq \frac{C \vert s \vert^M e^{ -(M+\frac{1}{2})} e^{ -\frac{ \pi}{2} \vert \tau \vert }}{ \vert s(s-1) \ldots (s-2M) \vert }.
 \]
 This guarantees the integrability of the modulus of the function $\Gamma(-s) \zeta(s) \cos( \frac{\pi}{2} s ) x^s $ on the line $\Re s = M+ \frac{1}{2}$ as $ \vert \zeta(s) \vert \leq 2 $,  $\vert \cos( \frac{\pi}{2} s ) \vert \leq e^{ \frac{ \pi}{2} \vert \tau \vert } $ and so the modulus of the integrand is less than
 \[
 \frac{ C \vert s \vert^M e^{-(M+ \frac{1}{2})}x^{M + \frac{1}{2}}}{ \vert s(s-1) \ldots (s-2M) \vert} \leq C_M \frac{1}{ \vert \tau \vert^{M+1}},
 \] $C_M $  being a positive constant, independent of $M$.
 In the same vein we can obtain an upper bound  of  $\displaystyle \frac{ \vert s \vert^M}{\vert s(s-1)\ldots (s-2M) \vert }$ on $\Re s= M+ \frac{1}{2}$ by grouping $(s-1)$ and $(s-2M)$,  $(s-2) $ and $(s-(2M-1)) \ldots $,  $(s-M)$ and $(s-(M+1)) $ to obtain
 \[
  \frac{ \vert s \vert^M}{\vert s(s-1)\ldots (s-2M) \vert } = \frac{ \vert s\vert^{M-1}}{(\tau^2+(M-1)^2) \ldots (\tau^2+\frac{1}{4} )} \leq \frac{\vert \tau \vert^{M-1}}{\vert \tau \vert^{2M}} = \frac{1}{ \vert \tau \vert^{M+1}}.
  \]
  We thus obtain the absolute convergence of the integral on $\Re s = M+\frac{1}{2}$ and, in the same way, the limit to zero of the integrals on the horizontal segments by using a  majorization, uniform  in $\sigma \in [\sigma_0, M+ \frac{1}{2} ]$,  of $\displaystyle \frac{ \vert s \vert^M}{\vert s(s-1)\ldots (s-2M) \vert }$.
   \subsection{Evaluation of the residues}
  \begin{enumerate}
  \item{ For $n \geq 2$}. 
  The poles of $\Gamma(-s) $ greater than $\sigma_0$ are  $s=n, n \geq 1$ and are of respective residues  $ \displaystyle \frac{(-1)^n}{n!}$. We find that if $n$ is odd the residue of $\Gamma(-s) \zeta(s) \cos( \frac{\pi}{2} s) x^s$ at $n$ vanishes, and if  $n= 2 l$ is even, the residue is $\displaystyle -(-1)^l \frac{\zeta(2l)}{2 l!} x^{2 l}$.
  \item For $n=1$ we have a double pole and by computing \[\lim_{ s\to 1} \frac{d}{ds}\big(  (s-1)^2 \Gamma(-s) \zeta(s) \cos( \frac{\pi}{2} s) x^s\big)\]   we find that the residue at $1$ is $-\pi x$.
  \end{enumerate}
The intermediate result  we obtained is the equality
\begin{align*}
 -\pi x  - \sum_{1 \leq l \leq M} (-1)^l \frac{ \zeta(2 l)}{2 l!} x^{2 l} = &\frac{1}{2i \pi} \int_{ \sigma_0-i \infty}^{ \sigma_0+ i \infty} \Gamma(-s) \zeta(s) \cos( \frac{\pi}{2} s) x^s ds\\
  &-  \frac{1}{2i \pi} \int_{ M+\frac{1}{2}-i \infty}^{ M+ \frac{1}{2}+ i \infty} \Gamma(-s) \zeta(s) \cos( \frac{\pi}{2} s) x^s ds.
\end{align*}
We need to analyze $\displaystyle I_ M(x)= \frac{1}{2i \pi} \int_{ M+\frac{1}{2}-i \infty}^{ M+ \frac{1}{2}+ i \infty} \Gamma(-s) \zeta(s) \cos(\frac{\pi}{2} s) x^s ds$. We
  write $\displaystyle \cos(\frac{\pi}{2} s)= \frac{  (e^{i\frac{\pi}{2}})^s + (e^{-i\frac{\pi}{2}})^s}{2} $ and making the legitimate  interchange of summation and integration we get
 \[
 I_M(x) = \frac{1}{2} \sum_{n \geq 1} \left ( \frac{1}{2i \pi} \int_{M+\frac{1}{2}- i \infty}^{M+\frac{1}{2} + \infty} \Gamma(-s) \Big( \frac{ e^{i \frac{\pi}{2}}x}{n}\Big )^s ds 
  +\frac{1}{2i \pi} \int_{M+\frac{1}{2}- i \infty}^{M+\frac{1}{2} + \infty} \Gamma(-s) \Big( \frac{ e^{-i \frac{\pi}{2}}x}{n}\Big )^s ds \right ) ,
  \]
  or
  \[
  I_M(x)= \frac{1}{2} \sum_{ n \geq 1} \left( J_M( \frac{ix}{n})+ J_M( \frac{-ix}{n}) \right)
  \]
  with
  \[
  J_M(z) = \frac{1}{2i \pi} \int_{M+\frac{1}{2}-\infty}^{M + \frac{1}{2}+ \infty} \Gamma(-s) z^s ds=  \frac{1}{2i \pi} \int_{-(M+\frac{1}{2})-\infty}^{-(M + \frac{1}{2})+ \infty} \Gamma(s) z^{-s} ds.
  \]
  According to \cite{higher} (7.3, p.348), the inverse Mellin transform of $e^{-ias} \Gamma(s)$, with the  conditions $\vert \Re a \vert \leq \frac{ \pi}{2},\; -m < \Re s < 1-m,\; m=1,2, \ldots $, is the function $\displaystyle  e^{-t e^{ia}} - \sum_{0 \leq r \leq m-1} \frac{ (-te^{ia})^r}{r!} $.
  For $\displaystyle a =- \frac{\pi}{2} + i \log \frac{n}{x} $ we have $\displaystyle \vert \Re a \vert = \frac{\pi}{2} $,\,
  $\displaystyle e^{ia}= -i \frac{x}{n}$ and
  \[
  J_M (i \frac{x}{n} )= \sum_{r \geq M+1} (-i)^r \frac{x^r}{n^r r!},\quad J_M (-i \frac{x}{n}) = \sum_{r \geq M+1} (i)^r \frac{x^r}{n^r r!}.
  \]
  Thus
  \[
  I_M(x)= \sum_{ n \geq 1} \left ( \sum_{r \geq M+1} ( i^r +(-i)^r) \frac{x^r}{n^r r!} \right).
  \]
  Since
  \[
   \sum_{ n \geq 1} \sum_{r \geq M+1} \frac{ \vert x \vert^r}{n^r r!} \leq e^{\vert x \vert} \sum_{n \geq 1} \frac{1}{n^{M+1}} = \zeta(M+1) e^{\vert x \vert} <+ \infty
   \]
   and $\displaystyle \lim_{ M \to + \infty} \sum_{ r \geq M+1}  ( i^r +(-i)^r) \frac{x^r}{n^r r!} =0$ we get \[\lim_{M \to + \infty} I_M(x)=0,\]
   and finally
   \[
   \sum_{ l \geq 1} (-1)^l \frac{ \zeta(2 l)}{2 l!} x^{2 l} = \frac{1}{2i \pi} \int_{ \sigma_0 - i \infty}^{ \sigma_0+ i \infty} \Gamma(-s) \zeta(s) \cos( \frac{\pi}{2} s) x^s ds +\, \pi x.
   \]
   We also see from the properties of the Fourier transform of an integrable function
 \begin{align*}
     \frac{1}{2i \pi} \int_{ \sigma_0 - i \infty}^{ \sigma_0+ i \infty} \Gamma(-s) \zeta(s)& \cos( \frac{\pi}{2} s) x^s ds\\
    &=\frac{x^{ \sigma_0}}{2 \pi} \int_\R \Gamma( - \sigma_0-i \tau) \zeta( \sigma_0+ i \tau) \cos ( \frac{ \pi}{2} (\sigma_0+ i \tau)) e^{i (\log x) \tau } d \tau \\
     &= x^{\sigma_0} o_{ \sigma_0}(x)
     \end{align*}
     If we replace $x$ by $\pi \nu \theta_{\nu},\,1\leq \nu \leq N$ and summing over $\nu$,  the terms coming from $\pi x$ disappear, since $\displaystyle \sum_{1\leq \nu \leq N}
     c_{\nu} \theta_{\nu}= 0$. This therefore gives the statement of the theorem.

   \end{proof}

\section{ Analytic continuation of ${\mathscr M}_s$ outside the unit disk.} 
  One of the purposes of this paragraph is to show that the unit circle is a natural boundary for $ {\mathscr M}_s $ for $ \Re s> 0 $. This can be easily done when $ s = k> 0 $ is an integer. Since 
 $ (\mu(n))_ {n \geq 1}=   (\mu_n)_ {n \geq 1}$ is a  non-periodic  sequence taking values in the finite set $ \{- 1,0,1 \} $ and the power series $ {\mathscr M}_0 $ is a non rational  function then, by Polya-Carlson theorem, the unit circle is a natural boundary. When $ k \geq 1 $ we can write
 \[ \left(z \frac{d}{dz}\right)^k {\mathscr M}_k= {\mathscr M}_0, \]
 showing that $ {\mathscr M}_0 $  is analytic where $ {\mathscr M}_k $ is analytic (compare with \eqref{iterated}). Hence the unit circle is a natural boundary for $ {\mathscr M}_k $ for every integer $k$. Alternatively we have
 \[{\mathscr M}_k (z) = \frac{1}{\Gamma k)} \int_0^ \infty {\mathscr M}_0 (e^{-t} z) t^{k-1} dt.\] 
 We introduce the operators $ T_s, \Re s> 0 $ defined on $ {\mathscr M}_0 $ by a Mellin's integral
 \[
  T_s({\mathscr M}_0)(z) = \frac{1}{\Gamma(s)} {\mathcal M} \left(  {\mathscr M}_0(e^{- \bullet} z)\right) (s), \quad \vert z \vert  <1.
  \]
  One can verifies that
  \[
  {\mathscr M}_s(z) = \frac{1}{\Gamma(s)} \int_0^\infty \sum_{ n \geq 1} \mu_n (e^{-t} z)^n t^{s-1} dt, \quad \vert z \vert  <1.
  \]
 Let $U\subset \BC $ be a star-shiped open set containing the origin. Let
 \[ {\mathcal O} _0 (U)=\left\{ f\,{\rm holomorphic\, on}\;  U,\, f(0)= 0 \right\} \]
 on which we define the family of operators
\[T_s(f)(z) =\frac{1}{\Gamma(s)} {\mathcal M}\left(f (e^{- \bullet} z)\right) (s), \quad z\in U,
 \]
 and show that it has a semi-group property (moreover holomorphic). 
  \begin{proposition} The family of operators $(T_s)_{{\Re}s >0}$ defined on ${\mathcal O}_0(U) $ by
 \[
 T_s(f)(z)= \frac{1}{ \Gamma (s)} \int_0^\infty f(e^{-t}z)t^{s-1} dt, \quad z \in U,
 \]
verifies
\[
T_{s_1+s_2}= T_{s_1} \circ T_{s_2}.
\]
\end{proposition}
Note  first that $\displaystyle  \lim_{t \to + \infty} f(e^{- t} z) = f (0) = 0 $ which ensures, for $ f \not \equiv 0 $, the existence of a integer $ N \geq 1 $ and a constant $ a_N \neq 0 $ 
such that $ f (e^{- t} z) = e^{- Nt} (a_N z^N + \cdots )$ for $t$ large and guarantees the convergence of the integral. One checks the holomorphicity of $ T_s (f) $ in $ U $ as usual. To show the  semi-group property we consider $ f \in {\mathcal O}_0 (U), z \in U, s_1, s_2 \in \BC,  {\Re} s_i> 0, i = 1,2 $, then
\begin{equation*} \begin{split}
T_{s_2}(T_{s_1} (f))(z) &= \frac{1}{\Gamma(s_2)} \int_0^\infty T_{s_1}(f)(e^{-t}z) t^{s_2-1} dt \\
&= \frac{ 1}{ \Gamma(s_2)} \int_0^\infty \left ( \frac{1}{ \Gamma(s_1} \int_0^\infty f(e^{-u} e^{-t}z)u^{s_1-1} du \right ) t^{s_2-1} dt \\
&= \frac{1}{ \Gamma (s_2) \Gamma(s_1)}\int  \int_{[0, \infty)^2} f(e^{-(u+t)}z) u^{s_1-1}t^{s_2-1} du \,dt.\\
\end{split} \end{equation*}
We set $t+u=v $,  $t-u=w$ and $\Delta= \{ (u,v): v \geq 0,\;  \vert w \vert  \leq v \}$ and obtain
\[
T_{s_2}(T_{s_1}(f))(z) = \frac{1}{2 \Gamma(s_2) \Gamma(s_1)} \int_0^\infty \left [ f(e^{-v}z) \int_{-v}^v \Big( \frac{ v-w}{2} \Big)^{s_1-1} \Big( \frac{ v+w}{2} \Big)^{s_2-1}  dw \right ]dv.
\]
The integral
\[
I = \int_{-v}^v \Big( \frac{ v-w}{2} \Big)^{s_1-1} \Big( \frac{ v+w}{2} \Big)^{s_2-1}  dw 
\]
is of Euler's type. We set $w= \rho v$ with $\vert \rho \vert \leq 1$ and obtain
\[
I = \frac{v^{s_1+s_2-1}}{2^{s_1+s_2-2}} \int_{-1}^1 (1-\rho^{s_1-1} (1+ \rho )^{s_2-1} d \rho.
\]
Now with $ x = 1+ \rho=2y$ we have
\begin{equation*} \begin{split}
\int_{-1}^1 (1-\rho^{s_1-1} (1+ \rho )^{s_2-1} d \rho & = \int_0^2 x^{s_2-1} (2-x)^{s_1-1} dx\\
&= 2^{s_1+s_2-1} \int_0^1 y^{s_2-1}(1-y)^{s_1-1} dy\\
&= 2^{s_1+s_2-1} \frac{ \Gamma(s_1) \Gamma(s_2)}{\Gamma(s_1+s_2)}  . \\
\end{split} \end{equation*}
That is
\[
I= 2 \frac{ \Gamma(s_1) \Gamma(s_2)}{\Gamma(s_1+s_2)} v^{s_1+s_2-1}
\]
and
\[
T_{s_2}(T_{s_1}(f))(z)= \frac{1}{\Gamma(s_1+s_2)} \int_0^\infty f(e^{-v}z)v^{s_1+s_2-1}dv= T_{s_1+s_2} (f)(z).
\]
The proposition is proved.\\ It is worth noting that this proof is contained in essence in \eqref{F-equation}, a representation of the Boole's differential operator
 $\displaystyle \vartheta= z\frac{d}{dz}$  acting on the $z$-variable, by a translation on the $s$-variable.
As a simple but illustrative example we take $U = D $ (or $\overline{D}\subset U$), $D$ being the unit disk. If $\displaystyle f(z) = \sum_{n \geq 1} a_n z^n,\, z\in D$, then 
\[T_s(f)(z) = \sum_{n \geq 1} \frac{a_n}{n^s} \] and \[T_{s_1+s_2} (f)(z)= \sum_{n \geq 1} \frac{a_n}{n^{s_1+s_2}} z^n\] is also equal to
\[T_{s_2}(T_{s_1}(f))(z)= \sum_{n \geq 1} \frac{1}{n_{s_2}} \big( \frac{a_n}{n^{s_1}} \big) z^n=   T_{s_1}(T_{s_2}(f))(z).\]

We can now finish the proof that the unit circle is a natural boundary for $ {\mathscr M}_s $.  If $ k>  {\Re} s $ is an integer, we write $ k = s + \sigma $ with $ \Re \sigma> 0 $. Assume that $ {\mathscr M}_s $ extends holomorphically  to an open set $ U $, containing $ D $, strictly larger than $D$. Without any loss of generality we can assume that $U$ is star-shiped with respect to the origin, then
\[{\mathscr M}_k = {\mathscr M}_ {s + \sigma} = T_ {s + \sigma} ({\mathscr M}_0) = 
T_{\sigma} ({\mathscr M}_s) \] extends to, which contradicts what have been said on the non holomorphic extendability of ${\mathscr M}_k(z)$.
  \section{conclusion}\label{lastsection}  
By way of conclusion we would like to come back to what was the motivation of this work, namely the Besicovitch question, and to mention that in fact it results from the identities of Kubert by means, in general, of a deep result of Number Theory as, for example, the Prime Number Theorem. We have also mentioned, very briefly, the occurrence of the Perron-Frobenius operator and the interpretation that can be drawn from it on the identities of Kubert. Our aim of this section is to point out the interest that there would be to put together  Number Theory, Harmonic Analysis and  Dynamical Systems in the study of arithmetic functions. The two twin functions of M\"{o}bius $\mu(n)$ and of Liouville $\lambda(n)$ share so many of these properties. For example, with the function \eqref{kubert} we have, as showed by \cite{D} and used in \cite{Sebbar}, \cite{Marcel}
\begin{align*}
\sum_{n=1}^{\infty} \frac{\mu(n)}{n} \{nt\}&= -\frac{1}{\pi} \sin 2\pi x,\\
\sum_{n=1}^{\infty} \frac{\lambda(n)}{n} \{nt\}&= -\frac{1}{\pi} \sum_{n=1}^{\infty} \frac{\sin2\pi n^2x}{n^2}.
\end{align*}
The second equality makes a link with what Riemann gave as an example of a continuous non-differentiable  function.  It is natural to define, similarly to ${\mathscr M}_s (z)$, the function
 \[{\mathscr N}_s (z)= \sum_{n\geq 1} \frac{\lambda(n)}{n^s} z^n. \]
 One of the main ideas of this work can be formulated in the following theorem and its corollary
 \begin{theorem}{\label{Main}}
 Let
 \[ {\mathfrak m}i_s(\theta)= \sum_{n=1}^{\infty} \frac{\mu(n)}{n^s}e^{2i\pi n \theta} ,\quad  \quad {\mathfrak n}i_s(\theta)= \sum_{n=1}^{\infty} \frac{\lambda(n)}{n^s} e^{2i\pi n \theta},\quad \theta\in \BR.\]
 Then for every positive $k$ we have
\[ \sum_{h=1}^k {\mathfrak m}i_s(h/k)= \frac{\mu(k)}{k^{s-1}} \sum_{n=1}^{\infty}\frac{\mu(n)}{n^s},\quad \quad  \sum_{h=1}^k {\mathfrak n}i_s(h/k)= \frac{\lambda(k)}{k^{s-1}} \sum_{n=1}^{\infty}\frac{\lambda(n)}{n^s}.\]
 \end{theorem}
 As we have seen this results from the following facts
\[ {\mathfrak m}i_s(h/k)=  \sum_{n=1}^{\infty}\frac{\mu(n)}{n^s}  \sum_{h=1}^k e^{2iphn/k}= \frac{\mu(k)}{k^{s-1}} \sum_{n=1, (n, k)=1}^{\infty} \frac{\mu(n)}{n^s}.\]
and 
\[\sum_{n=1, \,(n,k)=1}^{\infty} \frac{\mu(n)}{n}= \lim_{s\to 1^{+}}\sum_{n=1,\, (n, k)}^{\infty} \frac{\mu(n)}{n^s}\]
\[= \lim_{s\to 1^+} \prod_{p\; \nmid\;   k}(1-p^{-s})= \lim_{s\to 1^+} \left\{ \zeta(s)   \prod_{p\vert k}(1-p^{-s})     \right \}^{-1}= 0.\]
and, similarly for $\Re s>1  $
\[\sum_{n=1}^{\infty}\frac{\lambda(n)}{n^s}=          \frac{\zeta(2s)}{\zeta(s)},\]
\begin{equation}\label{lambda}
\sum_{n=1}^{\infty}\frac{\lambda(n)}{n}=  \lim_{s\to 1^+} \frac{\zeta(2s)}{\zeta(s)}= 0.
\end{equation}
\begin{corollary} The functions 
$ \Re{\mathfrak m}i_s(\theta),     \Re{\mathfrak n}i_s(\theta)  $                  
are non-trivial real-valued continuous functions 
$f$ on the real line which have period unity, are even, and for every positive integer $k$ have the property     \[ \sum_{h=1}^n f(h/k)=0.\] 
\end{corollary}
 Furthermore, one can prove directly that $\displaystyle  \sum(-1)^n \frac{\lambda(n)}{n}= 0  $. Indeed
\[ \sum_{n=1}^{\infty}(-1)^n \frac{ \lambda(n)}{n^s} + \sum_{n=1}^{\infty} \frac{\lambda(n)}{n^s} = 2\sum_{n=1}^{\infty} \frac{\lambda(2n)}{(2n)^s}.\]
Since $\lambda(2)=-1$ and  $  \lambda$ is multiplicative we obtain
\[\frac{\lambda(2n)}{(2n)^s} = -\frac{1}{2^s}  \frac{\lambda(n)}{n^s}\]
 so that
\[ \sum_{n=1}^{\infty} (-1)^n\frac{\lambda(n)}{n^s} = (-1-\frac{2}{2^s}) \sum_{n=1}^{\infty} \frac{\lambda(n)}{n^s },\]
or
\[ \sum_{n=1}^{\infty} (-1)^{(n+1)}\frac{\lambda(n)}{n^s} = (1+\frac{2}{2^s}) \frac{\zeta(2s)}{\zeta(s)}.\] 
We conclude by taking the limit $s\to 1$ as in \eqref{lambda}.\\
 It is generally believed that the values of the M\"{o}bius and  Liouville functions enjoy various randomness properties. One manifestation of this principle is an old conjecture of
Chowla \cite{Chowla}  asserting that for all $l\in \BN  $ and all distinct $n_1, n_2,\cdots,n_l\in \BN     $ and  for every $\displaystyle \epsilon_1,\epsilon_2,\cdots,\epsilon_l \in\{1,2 \}                           $
we have
\begin{align*}
\sum_{m=1}^M \mu^{\epsilon_1}(m+n_1)\cdots  \mu^{\epsilon_l}(m+n_l)&= o(M)\\
\sum_{m=1}^M \lambda^{\epsilon_1}(m+n_1)\cdots  \lambda^{\epsilon_l}(m+n_l)&= o(M).
\end{align*}
According to P. Sarnak we say that a sequence $\displaystyle a(n)_{n\geq}1$  is  deterministic if there exists a topological dynamical system (X, T) with zero topological entropy, a point $x\in X$, and a continuous function $f : X \to \C$ such that for all $n\geq 1$,
$a(n)= f(T^n(x))$.  Sarnak's conjecture  states that for every deterministic sequence   $a(n)_{n\geq}1$ we have
\[ \sum_{m=1}^M \mu(n) a(n)= o(M).          \]
The case of $X$ is a point corresponds to the estimate $\displaystyle \sum_{m=1}^M \mu(n)= o(M)$, an equivalent form of the Prime Number Theorem. When $X= \R/\Z$ and 
$T(x)= x+\alpha \,{\rm(modulo \,1)}$, Sarnak's conjecture results from Davenports's estimate. We refer to \cite{Cellarosi}  for an extended report on these 
innovative ideas.

\newpage
\begin{figure}[h]
\includegraphics[width=\textwidth]{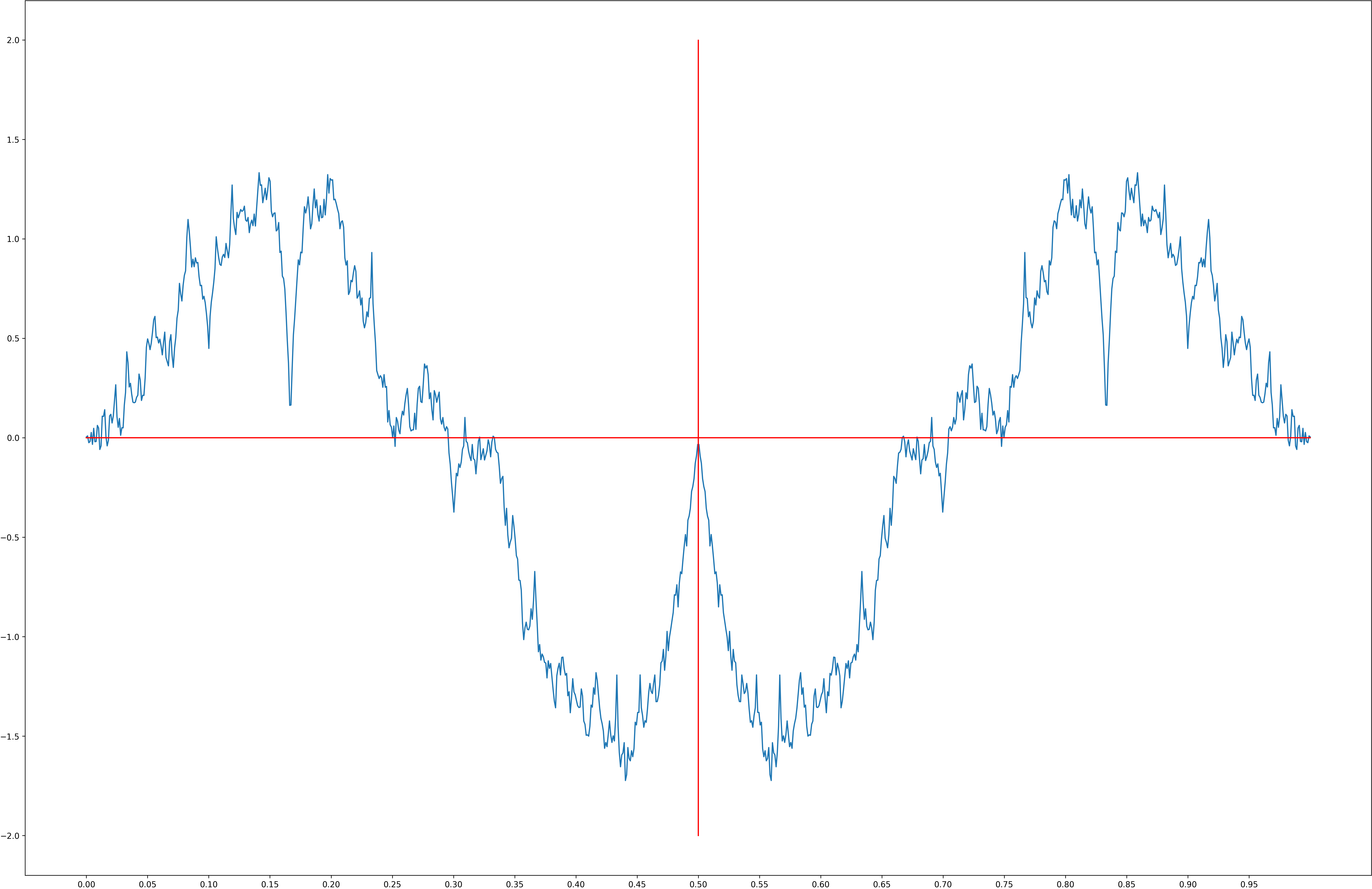}
\caption{Graph of $\sum \frac{\mu(n)}{n} \cos(2\pi nt)$}
\label{fig:mu}
\end{figure}

\newpage
\begin{figure}[h]
\includegraphics[width=\textwidth]{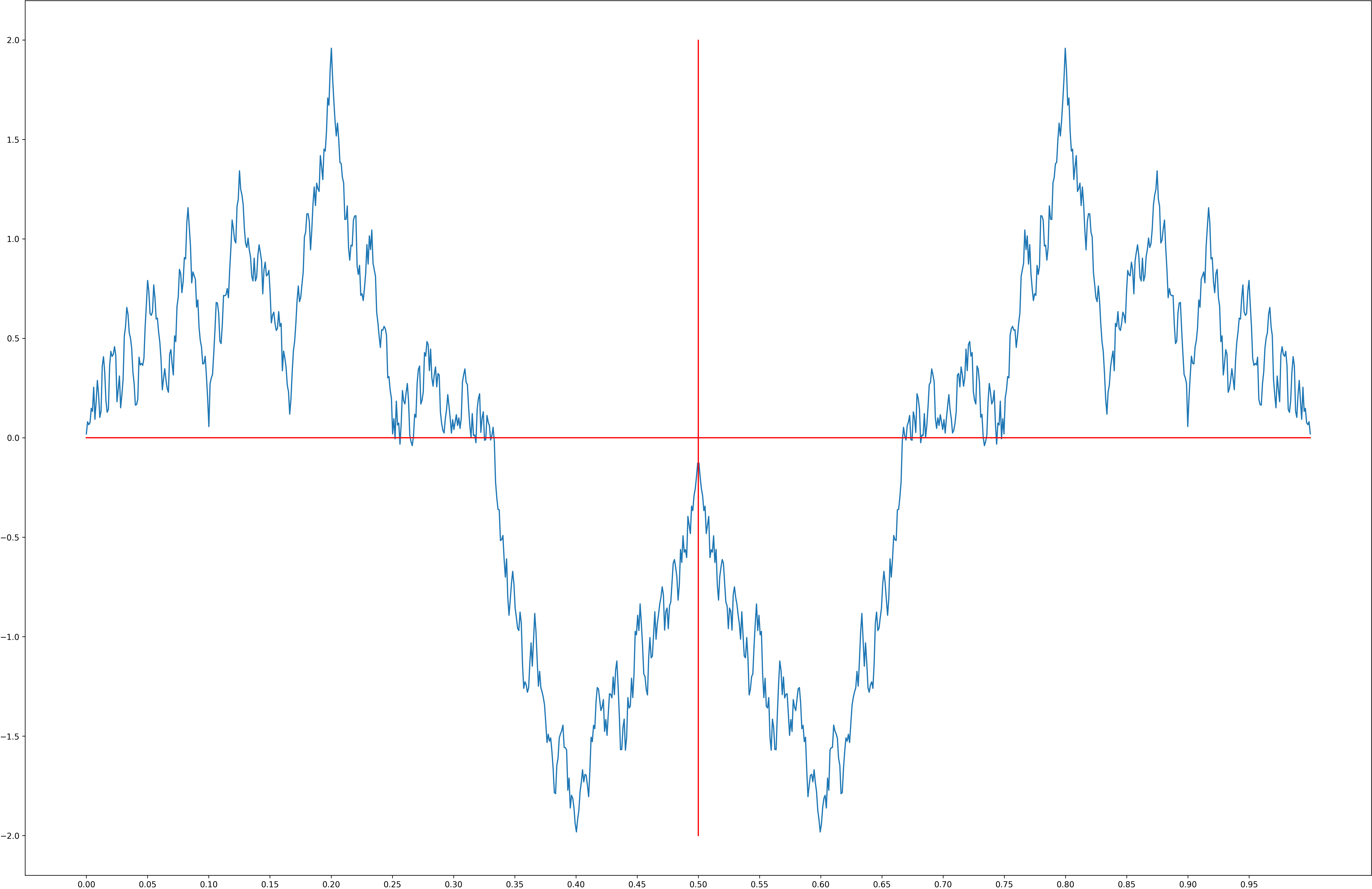}
\caption{Graph of $\sum \frac{\lambda(n)}{n} \cos(2\pi nt)$}
\label{fig:lambda}
\end{figure}

\end{document}